\documentclass[12pt]{amsart}

\usepackage[dvipdfmx]{graphicx}
\usepackage{cite}
\usepackage{amsthm, amsmath, amssymb, mathtools}
\usepackage[top=45truemm, bottom=45truemm, left=30truemm, right=30truemm]{geometry}

\renewcommand{\limsup}{\varlimsup}

\newtheorem{theorem}{Theorem}[section]
\newtheorem{lemma}[theorem]{Lemma}
\newtheorem{corollary}[theorem]{Corollary}

\newtheorem{proposition}[theorem]{Proposition}

\newtheorem{question}[theorem]{Question}

\theoremstyle{definition}
\newtheorem{remark}[theorem]{Remark}
\newtheorem{notation}[theorem]{Notation}

\newcommand{\subpartial}{\tilde{\partial}}
\newcommand{\Conn}{\mathrm{Conn}}
\newcommand{\ord}{\mathrm{ord}}
\numberwithin{equation}{section}

\title[Topology and algebraic independence of PRCs]{Topological  properties and algebraic independence of 
sets of prime-representing constants}
\author[K. Saito]{Kota Saito}
\date{}
\address{Kota Saito\\
Faculty of Pure and Applied Sciences, University of Tsukuba, Tsukuba, Japan}
\curraddr{}
\email{saito.kota.gn@u.tsukuba.ac.jp}

\author[W. Takeda]{Wataru Takeda}
\address{Wataru Takeda\\ Department of Applied Mathematics, Tokyo University of Science,
1-3 Kagurazaka, Shinjuku-ku, Tokyo 162-8601, Japan.}
\email{w.takeda@rs.tus.ac.jp}
\date{}
\curraddr{}

\subjclass[2020]{11B05, 11J68, 11J81}
\keywords{prime-representing functions, prime-representing constants, Mills' constant, Cantor set, transcendental numbers, algebraic independence, topology}

\begin{document}
\maketitle
\begin{abstract}
Let $(c_k)_{k=1}^\infty$ be a sequence of positive integers. We investigate the set of $A>1$ such that the integer part of $A^{c_1\cdots c_k}$ is always a prime number for every positive integer $k$. Let $\mathcal{W}(c_k)$ be this set. The first goal of this article is to determine the topological structure of $\mathcal{W}(c_k)$. Under some conditions on $(c_k)_{k=1}^\infty$, we reveal that $\mathcal{W}(c_k)\cap [0,a]$ is homeomorphic to the Cantor middle third set for some $a$. The second goal is to propose an algebraically independent subset of $\mathcal{W}(c_k)$ if $c_k$ is rapidly increasing. As a corollary, we disclose that the minimum of $\mathcal{W}(k)$ is transcendental. In addition, we apply the main result to $\mathcal{W}(c_k)$ in the case when $c_1\cdots c_k=3^{k!}$. 
As a consequence, we give an algebraically independent and countably infinite subset of this set. Furthermore, we also get results on the rational approximation, $\mathbb{Q}$-linear independence, and numerical calculations of elements in $\mathcal{W}(c_k)$.  
\end{abstract}
\section{Introduction}\label{Section-introduction}
Let $\mathbb{N}$ denote the set of all positive integers, and let $\lfloor x\rfloor$ denote the integer part of $x\in \mathbb{R}$. We say that $f: \mathbb{N} \rightarrow \mathbb{N}$ is a \textit{prime-representing function} (\textit{PRF}) if $f(k)$ is a prime number for every $k\in \mathbb{N}$. Let $(c_k)_{k\in \mathbb{N}}$ be any real sequence satisfying $c_1>0 $ and $c_k>1$ for all $k\geq 2$. We define 
\[
\mathcal{W}(c_k) =\{A>1\colon \lfloor A^{C_k} \rfloor \text{ is a PRF}  \},
\]  
where we define $C_k=c_1\cdots c_k$ for every $k\in \mathbb{N}$. Further, an element in $\mathcal{W}(c_k)$ is called a \textit{prime-representing constant} (\textit{PRC}) with respect to $(c_k)_{k\in \mathbb{N}}$. Mills was the first to propose a PRF of this type. He showed that there exists a constant $A>1$ such that $\lfloor  A^{3^k}\rfloor$ is a PRF \cite{Mills}. Therefore, $\mathcal{W}(3)$ is non-empty from his result. The minimum of $\mathcal{W}(3)$ is called Mills' constant. We will verify the existence of the minimum of $\mathcal{W}(3)$ in Remark~\ref{Remark-min}. Assuming the Riemann hypothesis, Caldwell and Cheng computed $600$ digits of the decimal expansion of Mills' constant \cite{CaldwellCheng}.  

Recall that a complex number $\alpha$ is called transcendental if $P(\alpha)\neq 0$ for all non-zero polynomials $P\in \mathbb{Q}[x]$. A finite set $\{\alpha_1,\ldots,\alpha_r\}\subseteq \mathbb{C}$ is called algebraically independent if $P(\alpha_1,\ldots , \alpha_r) \neq 0$ 
for all non-zero polynomials $P\in \mathbb{Q}[x_1,\ldots, x_r]$. Furthermore, an infinite set $Z\subseteq \mathbb{C}$ is called algebraically independent if all non-empty finite subsets of $Z$ are algebraically independent.

Since $\mathcal{W}(3)$ is uncountable by a result of Wight\cite[Section~6]{Wright54}, we can find uncountably many transcendental PRCs in $\mathcal{W}(3)$. However, we do not know which PRCs are transcendental. It is still open to determine whether Mills' constant is rational or irrational. Interestingly, Alkauskas and Dubickas gave results on the transcendence of PRCs in a different way \cite[Theorem~1]{AlkauskasDubickas}. 
Specifically, they constructed a transcendental number in $\mathcal{W}(c_k)$ if a real sequence $(c_k)_{k\in\mathbb{N}}$ satisfies that 
\begin{enumerate}
\item $c_1=1$;
\item $c_{k+1}>2.1053$ and $C_k\in \mathbb{N}$ for all $k\in \mathbb{N}$;
\item \label{Condition-Large-ck} $\limsup_{k\rightarrow \infty} c_{k+1}=\infty$. 
\end{enumerate}

As a corollary, they showed that we can choose a transcendental number $\xi>1$ such that $\lfloor \xi^{(k+1)!} \rfloor$ is a PRF \cite[Corollary]{AlkauskasDubickas}. However, they did not discuss whether the minimum of $\mathcal{W}(c_k)$ is transcendental. 

In this article, we present an improvement on their results. We discuss the algebraic independence of certain countably infinite subsets of $\mathcal{W}(c_k)$. Especially, we reveal that the minimum of $\mathcal{W}(c_k)$ is transcendental under some conditions $(c_k)_{k\in \mathbb{N}}$. Before stating the main theorem, we introduce the  sub-boundary. Let $\partial X$ denote the boundary of $X\subseteq \mathbb{R}$. Let $\mathrm{Conn}(Y)$ denote the family of connected components of $Y\subseteq \mathbb{R}$. We define
\begin{gather*}
\subpartial X=\bigcup_{U\in \mathrm{Conn}(\mathbb{R}\setminus X) }  \partial U,\\
\subpartial_L X=\{\sup U\in \mathbb{R}\colon U\in \mathrm{Conn}(\mathbb{R}\setminus X) \},\quad \subpartial_R X=\{\inf U\in \mathbb{R}\colon U\in \mathrm{Conn}(\mathbb{R}\setminus X) \},
\end{gather*}
for all $X\subseteq \mathbb{R}$. We say that $\subpartial X$ is the \textit{sub-boundary} of $X$, and we say that $\subpartial_L X$ and $\subpartial_R  X$ are the \textit{left sub-boundary} and the \textit{right sub-boundary} of $X$, respectively. Note that for all $U \in \Conn (\mathbb{R}\setminus X )$ we obtain $\partial U\subseteq \partial X$. Thus, $\subpartial X \subseteq \partial X$. Further, $\subpartial X= \subpartial_L X \cup \subpartial_R X$ since all connected subsets of $\mathbb{R}$ are intervals or singletons.

We now give an important example. Let  
\[
\mathcal{C} =\left\{\sum_{i=1}^\infty \frac{x_i}{3^{i}} \colon x_i\in \{0,2\} \text{ for all }i\in \mathbb{N}  \right\},
\] 
which is called the Cantor middle third set. Then we obtain
\begin{align*}
\subpartial \mathcal{C}&=\left\{ \sum_{i=1}^\infty \frac{x_i}{3^{i}} \colon x_i\in\{0,2\} \text{ for all $i\in \mathbb{N}$, and } x_N=x_{N+1}=\cdots \text{ for some $N\in\mathbb{N}$} \right\}.
\end{align*}
From the identity, we see that $\subpartial\mathcal{C}\subseteq \mathbb{Q}$. Therefore, we can find many algebraic relations among elements in $\subpartial \mathcal{C}$. Inspired by this fact, let us discuss the following question.

\begin{question}
Given a Cantor-like set $X$ (for example, a totally disconnected and closed set), can we find any algebraic relations among elements in $\subpartial X$? 
\end{question}

 The first goal of this article is to determine the topological structure of $\mathcal{W}(c_k)$. A topological space $X$ is called \textit{perfect} if $X$ is closed, and has no isolated points. Also, $X$ is called \textit{totally disconnected} if all non-empty connected subsets of $X$ are singletons.  
\begin{theorem}\label{Theorem-topology}
Let $(c_k)_{k\in \mathbb{N}}$ be a sequence of real numbers. Suppose that 
\begin{enumerate}
\item \label{Theorem-topology-1} $c_1> 0$;  
\item \label{Theorem-topology-2} $c_{k+1}\in \mathbb{Z}$ and $c_{k+1}\geq 2$ for all $k\in \mathbb{N}$;
\item \label{Theorem-topology-3} $c_{k+1}\geq 3$ for infinitely many $k\in \mathbb{N}$.
\end{enumerate}
Then $\mathcal{W}(c_k)$ is non-empty, totally disconnected, and perfect. 
\end{theorem}
We will prove Theorem~\ref{Theorem-topology} in Section~\ref{Section-Topology}. From this theorem, $\mathcal{W}(c_k)\cap [0, a]$ is homeomorphic to the Cantor middle third set $\mathcal{C}$ for every sufficiently large $a\in \mathbb{R}\setminus \mathcal{W}(c_k)$ if $(c_k)_{k\in \mathbb{N}}$ satisfies all the conditions in the theorem. We will give the details in Remark~\ref{Remark-Topology}.  

The second goal is to reveal the algebraic independence of subsets of $\subpartial \mathcal{W}(c_k)$.

\begin{theorem}\label{Theorem-main1}Let $r\in \mathbb{N}$. Let $(c_k)_{k\in\mathbb{N}}$ be a sequence of integers satisfying that
\begin{enumerate}
\item \label{Condition-c1}$c_1\geq 1$;
\item \label{Condition-c2} $c_{k+1}\geq 2$ for all $k\in \mathbb{N}$;
\item \label{Condition-Ck} $\limsup_{k\rightarrow \infty} c_{k+1} C_k^{1-r}=\infty.$
\end{enumerate}
Then for all $A_1\in \subpartial \mathcal{W}(c_k)$, $A_2,\ldots , A_r\in \subpartial_L\mathcal{W}(c_k)$ with $A_1<A_2<\cdots <A_r$, the set 
\[
\{A_1,\ldots, A_r\}
\]
 is algebraically independent. In particular, all elements in $\subpartial \mathcal{W}(c_k)$ are transcendental. Especially, the minimum of $\mathcal{W}(c_k)$ is transcendental. 
\end{theorem}

\begin{remark}In the case when $r=1$, Theorem~\ref{Theorem-main1} states that $\{A_1\}$ is algebraically independent for all $A_1\in \subpartial \mathcal{W}(c_k)$, which means that $A_1$ is transcendental.
\end{remark}

\begin{remark}\label{Remark-sub-boundary}
Let $(c_k)_{k\in \mathbb{N}}$ be a sequence of real numbers satisfying the conditions in Theorem~\ref{Theorem-main1}. Then $c_{k+1}\geq 3$ for infinitely many $k\in \mathbb{N}$ by \eqref{Condition-Ck} in Theorem~\ref{Theorem-main1}. Thus, $\mathcal{W}(c_k)$ is non-empty, totally disconnected, and perfect by Theorem~\ref{Theorem-topology}. Hence, $\Conn (\mathbb{R}\setminus \mathcal{W}(c_k))$ is a family of countably infinitely many open intervals. This yields that $\subpartial_L \mathcal{W}(c_k)$ and $\subpartial_R \mathcal{W}(c_k)$ are countably infinite. Further, since $\mathcal{W}(c_k)$ is closed, 
\[
\subpartial_L \mathcal{W}(c_k)\cup \subpartial_R \mathcal{W}(c_k)=\subpartial \mathcal{W}(c_k) \subseteq \partial \mathcal{W}(c_k)\subseteq \mathcal{W}(c_k).
\]  
By this  argument, all elements $A_1,\ldots, A_r$ in Theorem~\ref{Theorem-main1} belong to $\mathcal{W}(c_k)$. Roughly speaking, if $c_k$ is rapidly increasing, Theorem~\ref{Theorem-main1} gives us a countably infinite subset of $\mathcal{W}(c_k)$ which is algebraically independent. 
\end{remark}

\begin{corollary}\label{Corollary-W(k)}
Each $A$ belonging to
\[
 \subpartial \mathcal{W}(k)=\subpartial \left( \{A>1\colon \lfloor A^{k!} \rfloor \text{ is a prime number for every $k\in \mathbb{N}$} \}\right)
 \]
is transcendental. In particular, the minimum of $\mathcal{W}(k)$ is transcendental.   
\end{corollary}

\begin{proof}
Let $c_k:=k$ for all $k\in \mathbb{N}$. Then $(c_k)_{k\in \mathbb{N}}$ satisfies \eqref{Condition-c1},
\eqref{Condition-c2}, and \eqref{Condition-Ck} with $r=1$ in Theorem~\ref{Theorem-main1}. Therefore, we obtain the corollary.  
\end{proof}

In Proposition~\ref{Proposition-linearly-ind}, we will state that for every $A_1\in \subpartial \mathcal{W}(k)$ and $A_2 \in \subpartial_L \mathcal{W}(k)$ with $A_1<A_2$, the set $\{A_1, A_2\}$ is linearly independent over $\mathbb{Q}$. In Proposition~\ref{Proposition-k!}, we will compute that the minimum of $\mathcal{W}(k)$ is $2.24199\cdots$ under the Riemann hypothesis.   

\begin{corollary}\label{Corollary-W(3k!)}
The set 
\[
\subpartial_L \left(\{A>1\colon \lfloor A^{3^{k!}} \rfloor \text{ is a prime number for every $k\in \mathbb{N}$}   \}\right) 
\]
is algebraically independent.
\end{corollary}

\begin{proof}Let $c_1=3^1$ and $c_k=3^{k!-(k-1)!}$ for all $k\geq 2$. Then $C_k=c_1\cdots c_k=3^{k!}$. Fix any $r\in \mathbb{N}$. It can be observed that
\begin{align*}
\frac{\log (c_{k+1} C_k^{1-r})}{\log 3}= k\cdot k! - (r-1) \cdot k!  = k!(k-(r-1))
 \rightarrow \infty  \quad \text{(as $k\rightarrow \infty$)}.
\end{align*}
Thus, $c_{k+1} C_k^{1-r}\rightarrow \infty$ as $k\rightarrow \infty$. Hence, $(c_k)_{k\in\mathbb{N}}$ satisfies \eqref{Condition-c1}, \eqref{Condition-c2}, and \eqref{Condition-Ck} in Theorem~\ref{Theorem-main1}. Therefore, $\subpartial_L \mathcal{W}(c_k)$ is algebraically independent by Theorem~\ref{Theorem-main1}.  
\end{proof}
Let $(c_k)_{k\in\mathbb{N}}$ be as in the proof of Corollary~\ref{Corollary-W(3k!)}. In Proposition~\ref{Proposition-3k!}, we will unconditionally compute that the minimum of $\mathcal{W}(c_k)$  is $1.30529\cdots$. We will also compute several PRCs in $\subpartial_L \mathcal{W}(c_k)$.

 We lastly give historical notes on the geometry of $\mathcal{W}(c_k)$. The geometric properties on $\mathcal{W}(c_k)$ were first studied by Wright. Based on his results \cite[Theorem 5,6,7, Section 6]{Wright54}, for all real sequences $(c_k)_{k\in \mathbb{N}}$ with $c_k\geq 2.658$ for all $k\in \mathbb{N}$, the set $\mathcal{W}(c_k)$ is uncountable, nowhere dense and has Lebesgue measure $0$. The admissible lower bounds for $c_k$ are closely related to problems on the distribution of prime numbers in short intervals. By a construction of Mills, if $\gamma>0$ satisfies that the interval $[x,x+x^\gamma]$ contains at least one prime number for every sufficiently large $x>0$, then the set $\mathcal{W}(c_k)$ is non-empty for all real sequences $(c_k)_{k\in \mathbb{N}}$ with 
 \[
 c_k \geq 1/(1-\gamma) \quad (k\in \mathbb{N}).
 \]
  Assuming the Riemann hypothesis\footnote{Actually, it is enough to assume the density hypothesis to obtain $\gamma=1/2+\epsilon$. Assuming the Riemann hypothesis, Cram\'er proved that there exists $B>0$ such that the interval $[x, x+B x^{1/2}\log x ]$ contains a prime number for every large $x>0$ \cite{Cramer}.}, $\gamma$ can be taken as $1/2+\epsilon$.   Therefore, we expect that the admissible lower bound of $c_k$ would be $2+\epsilon$. Matom\"aki gave an interesting improvement. Surprisingly, she could unconditionally show that the result of Wright is still true if we replace $c_k\geq 2.658$ with $c_k\geq 2$ \cite[Theorem~3]{Matomaki}. Recently, the first author of this article studied the fractal geometric structure of $\mathcal{W}(c_k)$ and showed that $\mathcal{W}(c_k)$ has Hausdorff dimension $1$ if $c_k\geq 2$ is bounded in $k$\cite{Saito}.

\begin{notation}
Let $\mathbb{N}$ be the set of all positive integers, $\mathbb{Z}$ be the set of all integers, $\mathbb{Q}$ be the set of all rational numbers, $\mathbb{R}$ be the set of all real numbers, and $\mathbb{C}$ be the set of all complex numbers. For all sets $X$, let us denote as $\# X$ the cardinality of $X$ and as $\mathcal{P}$ the set of all prime numbers.  Also, we define $[\ell]=\mathbb N\cap[1,\ell]$.
\end{notation}

\section{Outline of the article and sketch of the proof of Theorem~\ref{Theorem-main1}}

Define $\theta=21/40$, and $g(x)=x/\log x$ for all $x>1$. We start with the following two theorems.

\begin{theorem}
\label{Lemma-BakerHarmanPintz} There exists a constant $d_0>0$ such that
\[
 \# ([x, x+x^{\theta}] \cap \mathcal{P})  \geq d_0 g(x^\theta)
 \]
for sufficiently large $x>0$.
\end{theorem}
\begin{proof}
See \cite{BakerHarmanPintz}.
\end{proof}

\begin{theorem}\label{Lemma-Matomaki}
There exist positive constants $d_1 < 1$ and $D$ such that, for every
sufficiently large $x$ and every $\gamma \in [1/2, 1]$, the interval $[x, 2x]$ contains at most
$Dx^{2/3-\gamma}$ disjoint intervals $[n, n + n^\gamma]$ for which
\[
\# ([n,n+n^\gamma] \cap \mathcal{P})  \leq d_1 g(n^\gamma).
\]
\end{theorem}
\begin{proof}
See \cite[Lemma~9]{Matomaki} and the proof of \cite[Lemma~1.2]{Matomaki2007}.
\end{proof}

The first goal is to determine the topological structure of $\mathcal{W}(c_k)$. By combining Theorem~\ref{Lemma-BakerHarmanPintz}, Theorem~\ref{Lemma-Matomaki}, and the construction method in \cite{Matomaki}, we will provide several lemmas in Section~\ref{Section-Matomaki}. Further, by these lemmas, we will disclose that $\mathcal{W}(c_k)$ is non-empty, totally disconnected, and perfect in Section~\ref{Section-Topology}. 

After Section~\ref{Section-UpperBounds}, we will discuss the algebraic independence of subsets of $\subpartial \mathcal{W}(c_k)$. Let us here give a sketch of the proof of Theorem~\ref{Theorem-main1}. Let $(c_k)_{k\in \mathbb{N}}$ be a sequence of integers satisfying the conditions in Theorem~\ref{Theorem-main1}. Let $r$ be a positive integer. Let $A_1\in \subpartial\mathcal{W}(c_k)$ and $A_2 ,\ldots,A_r \in \subpartial_L \mathcal{W}(c_k)$. Suppose that $A_1<\cdots <A_r$. Take any non-zero polynomial $P\in \mathbb{Z}[x_1,\ldots,x_r ]$. Let $p_j(k)=\lfloor A_j^{C_k}\rfloor$. For all $k\in \mathbb{N}$, let 
\[
\alpha_1=\alpha_1(k)=
\begin{cases}
(p_1(k))^{1/C_k}&\quad \text{if $A_1\in \subpartial_L \mathcal{W}(c_k)$},\\
(p_1(k)+1)^{1/C_k}&\quad \text{if $A_1\in \subpartial_R \mathcal{W}(c_k)$},  
\end{cases}
\]
and $\alpha_j=\alpha_j(k)=p_j(k)^{1/C_k}$ for all $2\leq j\leq r$ . Since $\mathcal{W}(c_k)$ is closed by Theorem~\ref{Theorem-topology},  $\subpartial \mathcal{W}(c_k)\subseteq \mathcal{W}(c_k)$. Thus, all the integers $p_j(k)$ are prime numbers. By the triangle inequality, 
\begin{equation}\label{Inequality-P(A1,A2,...,Ar)}
 |P(A_1,\ldots , A_r)| \geq |P(\alpha_1,\ldots , \alpha_r)| -|P(A_1,\ldots , A_r)- P(\alpha_1,\ldots , \alpha_r)|=: S-T.
\end{equation}
We will prove that $T\leq S/2 $ for some large $k\in \mathbb{N}$. 
Assuming the above to be true, then one has $|P(A_1,\ldots , A_r)|\geq S/2>0$. This yields that $\{A_1,\ldots, A_r\}$ is algebraically independent.

In order to evaluate the upper bounds for $T$ in \eqref{Inequality-P(A1,A2,...,Ar)}, we will observe that each $A_j\in \subpartial \mathcal{W}(c_k)$ is very near $\alpha_j$. Actually, there exist constants $\eta_j>0$ and $\gamma_j>1$ depending only on $A_j$ and $c_1$ such that
\[
|A_j -\alpha_j(k)|\leq \eta_j \cdot  \gamma_j^{(\theta-1)C_{k+1}}
\]
for all $j\in [r]$ and for infinitely many $k\in \mathbb{N}$. We will prove this inequality in Section~\ref{Section-UpperBounds}. Using this inequality and the mean value theorem, we obtain the upper bounds for $T$. 

 On the other hand, since all the integers $p_j(k)$ are prime numbers, the set $\{\alpha_1(k), \ldots ,\alpha_r(k)\}$ is independent in the sense that 
\begin{equation*}\label{Inequality-outlineP}
S=|P(\alpha_1(k),\ldots , \alpha_r(k))|>0
\end{equation*}
if $k$ is sufficiently large. Furthermore, we will obtain the quantitative lower bounds for $S$ in Section~\ref{Section-proof}. These lower bounds are very small. However, by comparing bounds for $S$ and $T$ in \eqref{Inequality-P(A1,A2,...,Ar)}, we will reach the desired inequality $T\leq S/2$ if $k$ is sufficiently large.  The key idea is that the set $\{A_1,\ldots ,A_r\}$ inherits algebraic independence from $\{\alpha_1(k),\ldots, \alpha_r(k)\}$. This idea comes from \cite{AlkauskasDubickas}.

In Section~\ref{Section-UpperBounds}, we will provide lemmas to evaluate upper bounds for $T$ in \eqref{Inequality-P(A1,A2,...,Ar)}. In Section~\ref{Section-LowerBounds}, we will present lemmas to evaluate lower bounds for $S$. Combining them, in Section~\ref{Section-proof}, we will  complete the proof of Theorem~\ref{Theorem-main1}. Furthermore, in Section~\ref{Section-Further}, we also report results on the rational approximation, $\mathbb{Q}$-linear independence, and numerical calculations of elements in $\mathcal{W}(c_k)$.

\section{Matom\"aki's construction}\label{Section-Matomaki}. 

\begin{lemma}\label{Lemma-Matomaki2}
Let $c\geq 2$ and $B>2$ be real numbers with $2\leq c\leq B$. Let $0<\epsilon<2/3$ also be a real number. Let $X$ and $Y$ be positive real numbers with 
\[
X^{1/3+\epsilon} \leq Y-X\leq ((3/2)^{1/B}-1)X.
\]
 Suppose that there exists a constant $d_2>0$ such that $\#([X, Y]\cap \mathcal{P})\geq  d_2 g(Y-X). $
Then there exists $X_0=X_0(d_2,B,\epsilon)>0$ such that if $X\geq X_0$, then we can find distinct prime numbers $q_1, q_2\in [X , Y]$ satisfying 
that for all $j\in \{1,2\}$ 
\[
\#([q_j^{c}, q_j^{c} + q_j^{c-1}]\cap \mathcal{P})\geq  d_1 g(q_j^{c-1}). 
\]
\end{lemma}

\begin{proof}
Take a sufficiently large parameter $X_0=X_0(d_2,B,\epsilon)$. We assume that $X\geq X_0$. Let $R= [X , Y]\cap \mathcal{P}$. By Theorem~\ref{Lemma-Matomaki} with $x:=X^{c}$ and $\gamma:=1-1/c$, the interval $[X^c, 2X^c]$ contains at most $DX^{c(2/3- (1-1/c))}$ disjoint intervals of the form
\[
[q^{c}, q^{c}+q^{c-1} ]\quad (q\in R)
\]
for which $\# ([q^{c}, q^{c}+q^{c-1} ]\cap \mathcal{P}) \leq d_1 g(q^{c-1}).$ Here we observe that for all $q\in R$
\begin{align*}
&X^{c}\leq q^{c} < q^{c} + q^{c-1} \leq Y^{c} + Y^{c-1}\leq Y^c (1+X^{-1})\\
&=(X+(Y-X))^c (1+X^{-1})\leq X^c (1+ (Y-X)/X)^B (1+X^{-1}) \leq 2X^c 
\end{align*}
since $Y-X\leq ((3/2)^{1/B}-1)X$, $X\geq X_0$, and $X_0=X_0(d_2,B,\epsilon)$ is sufficiently large. Thus, all intervals $[q^{c}, q^{c}+q^{c-1}]$ $(q\in R)$ are contained in $[X^{c}, 2X^{c} ]$. Hence, at least 
\begin{equation}\label{Equation-section2-1}
d_2 g(Y-X)- DX^{c(2/3- (1-1/c))} \geq d_2 \frac{X^{1/3+\epsilon}}{\log X^{1/3+\epsilon} }-DX^{1-c/3}
\end{equation}
intervals $[q^{c}, q^{c}+q^{c-1}]$ $(q\in R)$ satisfy that $\#([q^{c}, q^{c} + q^{c-1}]\cap \mathcal{P})\geq  d_1 g(q^{c-1})$. We see that the right-hand side of \eqref{Equation-section2-1} is greater than or equal to $2$ since $1-c/3\leq 1/3<1/3+\epsilon$ and $X\geq X_0$. Hence, there exist distinct  $q_1,q_2 \in R$ such that $\#([q_j^{c}, q_j^{c} + q_j^{c-1}]\cap \mathcal{P})\geq  d_1 g(q_j^{c-1})$ for all $j\in\{1,2\}$.
\end{proof}

\begin{lemma}\label{Lemma-Matomaki3}
Let $(e_k)_{k=2}^\infty$ be a real sequence with $e_k\geq 2$ for all $k\in \mathbb{N}$. Let $B=1/(1-\theta)$. Let $X$ and $Y$ be positive real numbers with $X^{1/2} \leq Y-X\leq ((3/2)^B-1)X$. Assume that there exists $d_3>0$ such that  
\begin{equation}\label{Inequality-I1-Matomaki3}
\#([X, Y]\cap \mathcal{P}) \geq d_3 g(Y-X).  
\end{equation}
Then there exists a large $X_0(d_3)>0$ such that if $X\geq X_0$, then we can find sequences of prime numbers $(q_{k,1})_{k=1}^\infty$ and $(q_{k,2})_{k=1}^\infty$ such that $X\leq q_{1,1} <q_{1,2}\leq Y$ and for all $j\in \{1,2\}$ and $k\in \mathbb{N}$  
\begin{equation}\label{Inequality-Matomaki}
q_{k,j}^{e_{k+1}} \leq q_{k+1,\: j}< (q_{k,j}+1)^{e_{k+1}}-1 .   
\end{equation}
\end{lemma}

\begin{proof}
Let $d=\min(d_0,d_1)$, where $d_0$ and $d_1$ are defined in Theorem~\ref{Lemma-BakerHarmanPintz} and Theorem~\ref{Lemma-Matomaki}, respectively. We define 
\begin{equation*}
\eta_{k+1}=
\begin{cases}
\theta e_{k+1}\quad &\text{if }e_{k+1}\geq 1/(1-\theta),  \\
e_{k+1}-1\quad &\text{if }2\leq e_{k+1} < 1/(1-\theta) 
\end{cases}
\end{equation*}
for all $k\in \mathbb{N}$. If $e_{2}\geq 1/(1-\theta)$, then we take prime numbers $q_{1,1}, q_{1,2} \in [X,Y]\cap \mathcal{P}$ with $q_{1,1}<q_{1,2}$. Note that such two prime numbers exist by \eqref{Inequality-I1-Matomaki3} and $X\geq X_0$.  Then for all $j\in \{1,2\}$, by Theorem~\ref{Lemma-BakerHarmanPintz} with $x:=q_{1,j}^{e_{2}}$, we have 
\[
\#([q_{1,j}^{e_{2}},\ q_{1,j}^{e_{2}} + q_{1,j}^{\eta_2}]\cap \mathcal{P}) \geq d g(q_{1,j}^{\eta_2} ).
\]
If $2\leq e_{2}< 1/(1-\theta)$, then since we are supposing \eqref{Inequality-I1-Matomaki3},  by applying Lemma~\ref{Lemma-Matomaki2} with $c:= e_{2}$, $B:=1/(1-\theta)$, $\epsilon:=1/6$, $X:=X$, and $Y:=Y$, there exist $q_{1,1}, q_{1,2}\in [X,Y]\cap \mathcal{P}$ such that for every $j\in\{1,2\}$ one has
\[
\#([q_{1,j}^{e_2}, q_{1,j}^{e_{2}}+q_{1,j}^{\eta_2}]\cap \mathcal{P})\geq d g(q_{1,j}^{\eta_2}).
\]
 Fix any $j\in \{1,2\}$ and let $q_1=q_{1,j}$. Let us construct a sequence of prime numbers $(q_k)_{k=1}^\infty$ such that for all $k\in \mathbb{N}$
\begin{equation}\label{Formula-for-qk}
\#([q_{k}^{e_{k+1}}, q_{k}^{e_{k+1}}+q_{k}^{\eta_{k+1}}]\cap \mathcal{P})\geq d g(q_{k}^{\eta_{k+1}}),\quad q_{k+1}\in [q_{k}^{e_{k+1}}, q_{k}^{e_{2}}+q_{k}^{\eta_{k+1}}]\cap \mathcal{P}.
\end{equation}

 Assume that for $k\in \mathbb{N}$ there exists $q_{k}\in \mathcal{P}$ with $q_{k}\geq X_0$ such that one has $\#([q_{k}^{e_{k+1}}, q_{k}^{e_{k+1}}+q_{k}^{\eta_{k+1}}]\cap \mathcal{P})\geq d g(q_{k}^{\eta_{k+1}}).$ Remark that $q_1$ satisfies this assumption. Let $X':=q_{k}^{e_{k+1}}$ and $Y':=q_{k}^{e_{k+1}}+q_{k}^{\eta_{k+1}}$. By the definition of $\eta_{k+1}$ and $X'=q_{k}^{e_{k+1}}$, 
\[
 Y'-X'=q_{k}^{\eta_{k+1}}\in \{X'^{\theta}, X'^{1-1/e_{k+1}} \} .
 \]
 Further, since $\theta=21/40$, $e_{k+1}\geq 2$, $q_{k}\geq X_0$, and $X_0$ is sufficiently large, we have 
\begin{gather*}
X'^{1/2}< X'^{\theta} \leq ((3/2)^{1/B}-1) X',\\
X'^{1/2} \leq  X'^{1-1/e_{k+1}} \leq q_{k}^{-1}  X'\leq {X_0}^{-1} X' \leq  ((3/2)^{1/B}-1) X'.
\end{gather*}
Thus,
$
X'^{1/2}\leq Y'-X'\leq ((3/2)^{1/B}-1) X'.
$
Therefore, by applying the argument when $k=1$ by substituting $X:=X'$ and $Y:=Y'$, there exists $q_{k+1}\in [q_{k}^{e_{k+1}}, q_{k}^{e_{k+1}}+q_{k}^{\eta_{k+1}}]\cap \mathcal{P}$ such that 
\[
\#([q_{k+1}^{e_{k+2}}, q_{k+1}^{e_{k+2}}+q_{k+1}^{\eta_{k+2}}]\cap \mathcal{P})\geq d g(q_{k+1}^{\eta_{k+2}}).
\]
By induction,  we construct a sequence of prime numbers $(q_k)_{k=1}^\infty(=(q_{k,j})_{k=1}^\infty)$ satisfying \eqref{Formula-for-qk}. Especially, for all $k\in \mathbb{N}$, we obtain $q_{k}^{e_{k+1}} \leq q_{k+1} \leq q_{k}^{e_{k+1}} +q_{k}^{\eta_{k+1}}.$ Take any $k\in \mathbb{N}$. Since $\eta_{k+1}\leq e_{k+1}-1$, we have 
\[
q_{k}^{e_{k+1}} +q_{k}^{\eta_{k+1}} \leq q_{k}^{e_{k+1}} +q_{k}^{e_{k+1}-1}< (q_{k}+1)^{e_{k+1}}-1. 
\]
Therefore, we conclude that \eqref{Inequality-Matomaki} holds.
\end{proof}

\begin{theorem}\label{Theorem-non-empty}
Let $(c_k)_{k\in \mathbb{N}}$ be a sequence of real numbers satisfying that 
\begin{enumerate}
\item $c_1>0$; 
\item $c_{k+1}\geq 2$ for all $k\in \mathbb{N}$. 
\end{enumerate}
Then $\mathcal{W}(c_k)$ is non-empty. 
\end{theorem}

\begin{proof}
Let $M\in \mathbb{N}$ be a sufficiently large parameter. Let $E=(3/2)^B$, where $B=1/(1-\theta)$. By the prime number theorem, we have
$\#([M, EM]\cap \mathcal{P}) \gg g(M)$. Here it follows that $M^{1/2} \leq EM-M \leq ((3/2)^B-1)M$. Therefore, by Lemma~\ref{Lemma-Matomaki3} with $(e_k)_{k=2}^\infty:=(c_k)_{k=2}^\infty$, $X:=M$, and $Y:=EM$, if $M$ is sufficiently large, then there exists a sequence of prime numbers $(p_k)_{k=1}^\infty$ such that for all $k\in \mathbb{N}$, we have $p_k^{c_{k+1}} \leq p_{k+1} < (p_k+1)^{c_{k+1}} -1$. This implies that 
\begin{equation}\label{Inequality-p1p2p3}
p_1^{1/C_1} \leq p_2^{1/C_2} \leq \cdots < (p_2+1)^{1/C_2} <(p_1+1)^{1/C_1}.
\end{equation}
Therefore, let $A= \lim_{k\rightarrow \infty} p_k^{1/C_k}$ which exists. Hence by \eqref{Inequality-p1p2p3}, for all $k\in \mathbb{N}$, we have $\lfloor A^{C_k}\rfloor=p_k\in \mathcal{P}$. This means that $A\in \mathcal{W}(c_k)$.  
\end{proof}

\begin{remark} 
We assume $c_1>0$ in Theorem~\ref{Theorem-non-empty}, but Matom\"aki assumed $c_1\geq 2$ in \cite[Theorem~3]{Matomaki}. Thus, there is a small gap between these assumptions. However, this gap is not essential. Indeed, let $(c_k)_{k\in \mathbb{N}}$ be a real sequence satisfying the conditions in Theorem~\ref{Theorem-non-empty}. Let $c'_1=2$ and $c'_k=c_k$ for all integers $k\geq 2$. Define 
\[
f:\mathcal{W}(c_k) \ni A \mapsto A^{c_1/2} \in \mathcal{W}(c_k').
\]
Then $f$ is Lipschitz continuous on every compact subset of $\mathcal{W}(c_k)$, and bijective. Since $\mathcal{W}(c_k')$ is uncountable, nowhere dense, and has Lebesgue measure $0$ by \cite[Theorem~3]{Matomaki}, $\mathcal{W}(c_k)$ also has the same properties by using $f$.  Hence, \cite[Theorem~3]{Matomaki} implies Theorem~\ref{Theorem-non-empty}.   
\end{remark}
\section{Topological properties}\label{Section-Topology}

In this section, we discuss the topological structure of $\mathcal{W}(c_k)$ and give a proof of Theorem~\ref{Theorem-topology}.

\begin{lemma}\label{Lemma-Pre1}
Let $(c_k)_{k\in \mathbb{N}}$ be a sequence of real numbers satisfying that 
\begin{itemize}
\item $c_1>0$; 
\item $c_{k+1}>1$ for all $k\in \mathbb{N}$. 
\end{itemize}
Suppose that $\mathcal{W}(c_k)$ is non-empty. Let $(A_j)_{j=1}^\infty$ be any sequence composed of $A_j\in \mathcal{W}(c_k)$ for all $j\in \mathbb{N}$. Then we have the following:
\begin{enumerate}
\item \label{Lemma-Pre1-1}if $A_1\geq A_2\geq \cdots>1$, then $\displaystyle{\lim_{j\rightarrow \infty}} A_j\in \mathcal{W}(c_k)$; \label{1-Lemma-Pre1}
\item if $A_1\leq A_2\leq \cdots<\infty$, and for all $k\in \mathbb{N}$ there exists an integer $\nu>k$ such that $c_{k+1}c_{k+2}\cdots c_\nu\in \mathbb{N}$, then $\displaystyle{\lim_{j\rightarrow \infty} A_j}\in \mathcal{W}(c_k)$.\label{2-Lemma-Pre1} 
\end{enumerate}
\end{lemma}

\begin{proof}
We firstly show \eqref{1-Lemma-Pre1}. Since $A_j$ is monotonically decreasing and has a lower bound, the limit $\lim_{j\rightarrow \infty}A_j$ exists. Let $A$ be this limit. Fix any $k\in \mathbb{N}$. Let $\epsilon= (1-\{A^{C_k}\})/2>0$, where $\{x\}$ is the fractional part of $x\in \mathbb R$. We have
\[
0\leq \{A^{C_k}\} + \epsilon = (1+\{A^{C_k}\})/2<1.
\]
Therefore, $\lfloor A^{C_k} \rfloor = \lfloor A^{C_k} +\epsilon\rfloor$. There exists a large $j\in \mathbb{N}$ such that
$0\leq A_j^{C_k} -A^{C_k} \leq \epsilon $, which implies $\lfloor A^{C_k}\rfloor \leq \lfloor A_j^{C_k}\rfloor\leq \lfloor A^{C_k}+\epsilon\rfloor = \lfloor A^{C_k} \rfloor$. Therefore, $\lfloor A^{C_k} \rfloor=\lfloor A_j^{C_k}\rfloor\in \mathcal{P}$.

Let us next show \eqref{2-Lemma-Pre1}. Similarly to the case \eqref{1-Lemma-Pre1}, the limit $\lim_{j\rightarrow \infty}A_j$ exists, and we let $A$ be this limit. Fix any $k\in \mathbb{N}$. We discuss the case when $A^{C_k}\in \mathbb{N}$. By the assumption, there exists an integer $\nu=\nu(k)>k$ such that $c_{k+1}\cdots c_\nu\in \mathbb{N}$. Let $D=c_{k+1}\cdots c_\nu$. Since $c_{k+1},\ldots ,c_\nu>1$ and $D\in \mathbb{N}$, we have $D\geq 2$. Here there exists a large $j\in \mathbb{N}$
such that $0\leq A^{C_kD}-A_j^{C_kD}<1/2$. Therefore $\lfloor A_j^{C_kD} \rfloor \in \{ A^{C_kD}-1, A^{C_kD}\}$ since $A^{C_k} $ and $D$ are positive integers. The case $\lfloor A_j^{C_kD} \rfloor=A^{C_kD}$ does not arise since $A^{C_k}\in \mathbb{N}$, $D\geq 2$ is an integer, and $\lfloor A_j^{C_kD} \rfloor$ is a prime number. In the case when $\lfloor A_j^{C_kD} \rfloor =  A^{C_kD}-1$,  one has $A^{C_k}=2$ since $\lfloor A_j^{C_kD} \rfloor$ is a prime number and $D\geq 2$ is an integer. Therefore, $\lfloor A^{C_k}\rfloor \in \mathcal{P}$. Hence, we may assume that $A^{C_k}$ is not an integer. Then $\lfloor A^{C_k} \rfloor < A^{C_k}$. Thus, for sufficiently large $j$, we have $\lfloor A^{C_k} \rfloor < A_j^{C_k}  \leq A^{C_k}<\lfloor A^{C_k} \rfloor+1 $. Therefore, $\lfloor A^{C_k}\rfloor =\lfloor A_j^{C_k}\rfloor\in \mathcal{P}$. We conclude that $A\in \mathcal{W}(c_k)$.    

\end{proof}

\begin{remark}\label{Remark-min}
By \eqref{Lemma-Pre1-1} in Lemma~\ref{Lemma-Pre1}, if $(c_k)_{k\in \mathbb{N}}$ satisfies the conditions in Lemma~\ref{Lemma-Pre1} and $\mathcal{W}(c_k)$ is non-empty, then the minimum of $\mathcal{W}(c_k)$ exists.    
\end{remark}

\begin{proposition}\label{Proposition-closed}
Let $(c_k)_{k\in \mathbb{N}}$ be a sequence of real numbers satisfying that
\begin{enumerate}
\item $c_1>0$; 
\item $c_{k+1}>1$ for all $k\in \mathbb{N}$;
\item for all $k\in \mathbb{N}$ there exists $\nu>k$ such that $c_{k+1}c_{k+2}\cdots c_\nu\in \mathbb{N}$.
\end{enumerate}
Then $\mathcal{W}(c_k)$ is closed. 
\end{proposition}

\begin{proof}
We may assume that $\mathcal{W}(c_k)$ is non-empty since the empty set is closed. Take any convergent sequence $(A_j)_{j=1}^\infty$ composed of $A_j\in \mathcal{W}(c_k)$ for all $j\in \mathbb{N}$. Let $A=\lim_{j\rightarrow \infty} A_j$. Then there exists a subsequence $j_1<j_2<\cdots \rightarrow \infty$ such that $A_{j_1} \leq  A_{j_2} \leq \cdots \rightarrow A$, or  $A_{j_1} \geq A_{j_2} \geq \cdots \rightarrow A$. Therefore, by Lemma~\ref{Lemma-Pre1}, we conclude that $A\in \mathcal{W}(c_k)$. 
\end{proof}

\begin{proposition}\label{Proposition-totally-disconnected}
Let $(c_k)_{k\in \mathbb{N}}$ be a sequence of real numbers satisfying that 
\begin{enumerate}
\item $c_k>0 $ for all $k\in \mathbb{N}$; 
\item \label{Condition-TD} $\limsup_{k\rightarrow \infty} C_k= \infty$ as $k\rightarrow \infty$.
\end{enumerate} 
Then $\mathcal{W}(c_k)$ is totally disconnected.
\end{proposition}
\begin{proof}
The empty set is totally disconnected. Thus, we may assume that $\mathcal{W}(c_k)$ is non-empty. Fix any $U\in \Conn (\mathcal{W}(c_k))$.  If $U$ contains an open interval $I$, then we take any $A\in U$. Let $p_k=\lfloor A^{C_k} \rfloor$ for all $k\in \mathbb{N}$. Since $A\in \mathcal{W}(c_k)$, all the integers $p_k$ are prime numbers. By the definition of $p_k$, we have $p_k\leq A^{C_k}<p_k+1$ for all $k\in \mathbb{N}$. Take a large $\ell\in \mathbb{N}$ such that $C_\ell>1$. Such $\ell$ exists by \eqref{Condition-TD}. Then there exists a composite number $m_\ell\in \mathbb{N}$ such that $m_\ell\in\{p_\ell+1, p_\ell+2\}$. Define $B_\ell=m_\ell^{1/C_\ell}$. Then we obtain $B_\ell\notin \mathcal{W}(c_k)$ since $\lfloor B_\ell^{C_\ell}\rfloor =m_\ell\notin \mathcal{P}$. Therefore, by the mean value theorem, 
\begin{equation}\label{Inequality-disconnected}
|B_\ell-A|=B_\ell-A\leq (p_\ell+2)^{1/C_\ell}-p_\ell^{1/C_\ell} \leq \frac{2}{C_\ell} p_\ell^{1/C_\ell-1} \leq \frac{2}{C_\ell}.  
\end{equation}
By \eqref{Condition-TD}, \eqref{Inequality-disconnected} contradicts that $A\in I \subseteq \mathcal{W}(c_k)$ and $I$ is open. 
\end{proof}

\begin{lemma}\label{Lemma-Mills} Let $(c_k)_{k\in \mathbb{N}}$ be a sequence of real numbers satisfying that 
\begin{enumerate}
\item $c_1>0$; 
\item $c_{k+1}\in \mathbb{Z}$ and $c_{k+1}\geq 2$ for all $k\in \mathbb{N}$. 
\end{enumerate}
Let $A\in \mathcal{W}(c_k)$, and let $p_k=\lfloor A^{C_k} \rfloor$ for all $k\in \mathbb{N}$. Then for all $k\in \mathbb{N}$, we have $p_{k}^{c_{k+1}}\leq  p_{k+1}< (p_{k}+1)^{c_{k+1}}-1.$
\end{lemma}

\begin{proof}
Take any $k\in \mathbb{N}$. Since $p_k=\lfloor A^{C_k}\rfloor$, it follows that $p_k \leq A^{C_k} < p_k+1$. This yields that 
$
p_k^{c_{k+1}} \leq  A^{C_{k+1}} < (p_k+1)^{c_{k+1}} .
$ 
Therefore, $p_{k+1} = \lfloor A^{C_{k+1}} \rfloor \geq p_k^{c_{k+1}}$ since $c_{k+1}$ and $p_k$ are positive integers. Further, by $(p_k+1)^{c_{k+1}}\in \mathbb{N}$ and the inequality $A^{C_{k+1}} < (p_k+1)^{c_{k+1}}$, we have 
$p_{k+1}\leq \lfloor A^{C_{k+1}} \rfloor \leq  (p_{k}+1)^{c_{k+1}}-1.$
Here we observe that
\[
(p_{k}+1)^{c_{k+1}}-1 = p_k (1+ (p_k+1)+\cdots+(p_k+1)^{c_{k+1}-1})
\]
is a composite number. Therefore, $p_{k+1}<(p_{k}+1)^{c_{k+1}}-1$.  
\end{proof}

\begin{proposition}\label{Proposition-no-isolated}
Let $(c_k)_{k\in \mathbb{N}}$ be a sequence of real numbers satisfying that 
\begin{enumerate}
\item $c_1>0$; 
\item $c_{k+1}\in \mathbb{Z}$ and $c_{k+1}\geq 2$ for all $k\in \mathbb{N}$;
\item $c_{k+1}\geq 3$ for infinitely many $k\in \mathbb{N}$.
\end{enumerate} 
Then $\mathcal{W}(c_k)$ has no isolated points.  
\end{proposition}

\begin{proof}
Let $\mathcal{I}=\{k\in \mathbb{N} \colon c_{k+1}\geq 3\}$. Fix any $A\in \mathcal{W}(c_k)$. Let $p_k=\lfloor A^{C_k} \rfloor$ for all $k\in \mathbb{N}$. Then by Lemma~\ref{Lemma-Mills}, for all $k\in \mathbb{N}$ we have
\begin{equation}\label{Inequality-noisolated}
p_{k}^{c_{k+1}}\leq  p_{k+1}< (p_{k}+1)^{c_{k+1}}-1.
\end{equation}
 Take a large $\ell\in \mathcal{I}$. By Theorem~\ref{Lemma-BakerHarmanPintz}, it follows that
\[
\#([p_\ell^{c_{\ell+1}}, p_\ell^{c_{\ell+1}}+ p_\ell^{\theta c_{\ell+1}}]\cap \mathcal{P})\geq d_0 g(p_\ell^{\theta c_{\ell+1}}).
\]
By Lemma~\ref{Lemma-Matomaki3} with $(e_2,e_3,\ldots):= (c_{\ell+2}, c_{\ell+3},\ldots)$, $X:=p_\ell^{c_{\ell+1}}$, $Y:=p_\ell^{c_{\ell+1}}+p_{\ell}^{\theta c_{\ell+1}}$, and $d_3:=d_0$, there exists a sequence $(q_k)_{k=\ell+1}^\infty$ such that 
\begin{equation}\label{Inequalities-twosequences}
q_{\ell+1}\in ([p_\ell^{c_{\ell+1}}, p_\ell^{c_{\ell+1}}+ p_\ell^{\theta c_{\ell+1}}]\cap \mathcal{P})\setminus \{p_{\ell+1}\}, \quad 
q_k^{c_{k+1}}\leq q_{k+1}<(q_{k}+1)^{c_{k+1}}-1
\end{equation}
 for all $k\geq \ell+1$. Remark that we can verify that $q_{\ell+1}\neq p_{\ell+1}$ from the existence of at least two different sequences of prime numbers $(q_{k,1})_{k=\ell+1}^\infty$ and $(q_{k,2})_{k=\ell+1}^\infty$ satisfying \eqref{Inequalities-twosequences}. Let $q_j=p_j$ for all $1\leq j\leq \ell$. Here we have $c_{\ell+1}\geq 3$ by $\ell \in \mathcal{I}$. Since $q_{\ell+1}\in [q_\ell^{c_{\ell+1}}, q_\ell^{c_{\ell+1}}+ q_\ell^{\theta c_{\ell+1}}]$ and $c_{\ell+1}\geq 3$, we obtain  
 \begin{equation}\label{Inequality-c-ell}
 q_\ell^{c_{\ell+1}}\leq q_{\ell+1}<(q_{\ell}+1)^{c_{\ell+1}}-1.
\end{equation} 
 Therefore, by \eqref{Inequality-noisolated}, \eqref{Inequalities-twosequences}, and \eqref{Inequality-c-ell}, we have 
\[
 q_{k}^{c_{k+1}}\leq  q_{k+1}< (q_{k}+1)^{c_{k+1}}-1
 \]
 for all $k\in \mathbb{N}$, which implies that
\begin{equation}\label{Inequality-q1q2q3}
q_1^{1/C_1}\leq q_2^{1/C_2} \leq \cdots <(q_2+1)^{1/C_2}  < (q_1+1)^{1/C_1}.   
\end{equation}
Let $A'=A'(\ell)=\lim_{k\rightarrow \infty} q_k^{1/C_k}$. This limit exists by \eqref{Inequality-q1q2q3}. In addition, $\lfloor A'^{C_k}\rfloor\in \mathcal{P}$ for all $k\in \mathbb{N}$. Thus, $A'(\ell)\in \mathcal{W}(c_k)$. Here it follows that $\lfloor A^{C_\ell} \rfloor =p_\ell = q_\ell =\lfloor A'^{C_\ell}\rfloor$. This implies that $|A-A'| \leq  (p_\ell+1)^{1/C_\ell} -p_\ell^{1/C_\ell}$. By the mean value theorem,
\[
|A-A'| \leq  (p_\ell+1)^{1/C_\ell} -p_\ell^{1/C_\ell} \leq \frac{1}{C_\ell} p_\ell^{1/C_\ell-1}\leq 2^{1-\ell}/c_1.
\]
Therefore, there exists a sequence $(A'(\ell))_{\ell\in \mathcal{I}}$ such that $\lim_{\ell \rightarrow \infty,\ \ell\in \mathcal{I}} A'(\ell)=A$, which means that $A$ is not an isolated point. 
\end{proof}

\begin{proof}[Proof of Theorem~\ref{Theorem-topology}] Let $(c_k)_{k\in\mathbb{N}}$ be a sequence of real numbers satisfying the conditions \eqref{Theorem-topology-1}, \eqref{Theorem-topology-2}, and \eqref{Theorem-topology-3} in Theorem~\ref{Theorem-topology}.  
Then $(c_k)_{k\in\mathbb{N}}$ satisfies all the conditions in Theorem~\ref{Theorem-non-empty}, Proposition~\ref{Proposition-closed}, Proposition~\ref{Proposition-totally-disconnected}, and Proposition~\ref{Proposition-no-isolated}. By combining them, $\mathcal{W}(c_k)$ is non-empty, totally disconnected, and perfect. 
\end{proof}

\begin{remark}\label{Remark-Topology} Let $(c_k)_{k\in\mathbb{N}}$ be a sequence of real numbers satisfying \eqref{Theorem-topology-1}, \eqref{Theorem-topology-2}, and \eqref{Theorem-topology-3} in Theorem~\ref{Theorem-topology}. For every sufficiently large  $a\in \mathbb{R}\setminus \mathcal{W}(c_k)$, $\mathcal{W}(c_k)\cap[0,a]$ is a non-empty, totally disconnected, perfect, and compact metric space. Indeed, by Theorem~\ref{Theorem-topology}, $\mathcal{W}(c_k)\cap[0,a]$ is non-empty, totally disconnected, closed, and compact. 
Let us verify that $\mathcal{W}(c_k)\cap [0,a]$ has no isolated points. Take any $A\in \mathcal{W}(c_k)\cap[0,a]$. Then we have $A<a$ since $a\notin \mathcal{W}(c_k)$. By Theorem~\ref{Theorem-topology}, $\mathcal{W}(c_k)$ has no isolated points. Therefore, there exists a sequence $(A_j)_{j=1}^\infty$ composed of $A_j\in \mathcal{W}(c_k)$ for all $j\in \mathbb{N}$ such that $\lim_{j\rightarrow \infty} A_j=A$. Hence by $A<a$, there is a large $j_0>0$ such that for all $j\geq j_0$ we have $A_j\in \mathcal{W}(c_k)\cap[0,a]$. This implies that $\mathcal{W}(c_k)\cap[0,a]$ has no isolated points. Furthermore, any two non-empty, totally disconnected, perfect, and compact metric spaces are homeomorphic (see \cite[30.3~Theorem]{Willard}). Therefore, $\mathcal{W}(c_k)\cap[0,a]$ is homeomorphic to the Cantor middle third set $\mathcal{C}$ for sufficiently large $a\in \mathcal{W}(c_k)\setminus \mathbb{R}$. 
\end{remark}

\section{Lemmas for evaluating upper bounds for $T$}\label{Section-UpperBounds}
In the following sections, we discuss the algebraic independence of $\subpartial \mathcal{W}(c_k)$. In this section, we provide lemmas to evaluate the upper bounds for $T$ in \eqref{Inequality-P(A1,A2,...,Ar)}.

\begin{lemma}\label{Lemma-key}
Let $(c_k)_{k\in \mathbb{N}}$ be a sequence of real numbers satisfying that  
\begin{itemize}
\item $c_1>0$;
\item $c_{k+1}\in \mathbb{Z}$ and $c_{k+1}\geq 2$ for all $k\in \mathbb{N}$; 
\item $c_{k+1}\geq 3$ for infinitely many $k\in \mathbb{N}$.
\end{itemize}
Fix any $A\in \subpartial \mathcal{W}(c_k)$. Let $p_k=\lfloor A^{C_k} \rfloor$ for all $k\in \mathbb{N}$. Let $\mathcal{I}=\{k\in\mathbb{N} \colon c_{k+1}\geq 3\}$. Then there exists $k_0=k_0(A)>0$ such that for all $k\in \mathcal{I}\cap [k_0,\infty)$, we have the following:
\begin{enumerate}
\item \label{Lemma-main-1} if $A\in \subpartial_L \mathcal{W}(c_k)$, then $p_{k}^{c_{k+1}} \leq p_{k+1} \leq p_{k}^{c_{k+1}} + p_{k}^{\theta c_{k+1}}$;
\item \label{Lemma-main-2} if $A\in \subpartial_R \mathcal{W}(c_k)$, then $(p_{k}+1)^{c_{k+1}}-(p_{k}+1)^{\theta c_{k+1}}  \leq p_{k+1} < (p_{k}+1)^{c_{k+1}} $.
\end{enumerate}
\end{lemma}

\begin{proof}
Take any $A\in \subpartial \mathcal{W}(c_k)$. By Proposition~\ref{Proposition-closed}, $A\in  \subpartial \mathcal{W}(c_k) \subseteq \partial \mathcal{W}(c_k)  \subseteq \mathcal{W}(c_k)$. Therefore, by applying Lemma~\ref{Lemma-Mills} to $A$, for all $k\in \mathbb{N}$, we have 
\begin{equation}\label{Inequality-pk}
p_{k}^{c_{k+1}}\leq  p_{k+1}< (p_{k}+1)^{c_{k+1}}-1. 
\end{equation}
By the definition of the sub-boundary, there exists a $U\in \Conn (\mathbb{R}\setminus \mathcal{W}(c_k))$ such that $A\in \partial  U$. Since $\mathbb{R}\setminus \mathcal{W}(c_k)$ is open, $U$ is an open interval. 

In the case when $A=\sup U$, meaning that $A\in \subpartial_L\mathcal{W}(c_k)$, we assume that there exists an infinite set $\mathcal{J}\subseteq \mathcal{I}$ such that 
$p_{k}^{c_{k+1}} +p_{k}^{\theta c_{k+1}} < p_{k+1}$ for all $k\in \mathcal{J}$. Take any large $\ell \in \mathcal{J}$. By Theorem~\ref{Lemma-BakerHarmanPintz},
\[
\#([p_{\ell}^{c_{\ell+1}},\ p_{\ell}^{c_{\ell+1}} + p_{\ell}^{\theta c_{\ell+1}}]\cap \mathcal{P}) \geq d_0 g(p_{\ell}^{\theta c_{\ell+1}}).
\]
We now apply Lemma~\ref{Lemma-Matomaki3} with $(e_2,e_3,\ldots):=(c_{\ell+2},c_{\ell+3},\ldots)$, $X:=p_{\ell}^{c_{\ell+1}}$, $Y:=p_{\ell}^{c_{\ell+1}} + p_{\ell}^{\theta c_{\ell+1}}$, and $d_3:=d_0$. Here $Y-X=p_{\ell}^{\theta c_{\ell+1}}=X^{\theta}=X^{21/40} \in [X^{1/2}, ((3/2)^B-1)X ]$ since $\ell$ is sufficiently large. Therefore, there exists a sequence $(q_k)_{k=\ell+1}^\infty$ of prime numbers such that $q_{\ell+1} \in [p_{\ell}^{c_{\ell+1}},\ p_{\ell}^{c_{\ell+1}} + p_{\ell}^{\theta c_{\ell+1}}]\cap \mathcal{P}$ and
for all integers $k\geq \ell+1$, we obtain
\begin{equation}\label{Inequalities-qkck+1} 
q_{k}^{c_{k+1}} \leq  q_{k+1} <(q_{k}+1)^{c_{k+1}} -1.
\end{equation}
Let $q_k=p_k$ for all $k\in [\ell]$. By $\ell \in \mathcal{J}\subseteq \mathcal{I}$, we have $c_{\ell+1}\geq 3> 1/(1-\theta)$, which yields that $ \theta c_{\ell+1} <c_{\ell+1}-1$. Therefore, 
\begin{equation}\label{Inequalities-qk1}
q_\ell^{c_{\ell+1}}\leq q_{\ell+1} \leq q_\ell^{c_{\ell+1}} +q_\ell^{\theta c_{\ell+1}}  <(q_\ell +1)^{c_{\ell+1}}-1.
\end{equation}
By combining \eqref{Inequality-pk}, \eqref{Inequalities-qkck+1}, and \eqref{Inequalities-qk1}, the inequality \eqref{Inequalities-qkck+1} holds for all $k\in \mathbb{N}$. This yields that 
\begin{equation}\label{Inequality-qks}
q_{1}^{1/C_{1}} \leq  q_{2}^{1/C_{2}} \leq \cdots < (q_{2}+1)^{1/C_{2}}<(q_{1}+1)^{1/C_{1}}.
\end{equation}
Let $A' =\lim_{k\rightarrow \infty} q_{k}^{1/C_k}$. Remark that this limit exists. Then we have $\lfloor A'^{C_k}\rfloor=q_{k}\in \mathcal{P}$ for all $k\in \mathbb{N}$, which implies that $A'\in \mathcal{W}(c_k)$.
Here recall that $A$ belongs to $\partial U$ for some open interval $U\in \Conn(\mathbb{R}\setminus \mathcal{W}(c_k))$. Thus, $U$ does not contain any elements in $\mathcal{W}(c_k)$. However, by $A^{C_{\ell}}< q_{\ell}+1$, \eqref{Inequality-qks}, the mean value theorem, and $c_{k+1}\geq 2$, 
\[
0< A-A' \leq (q_{\ell}+1)^{1/C_{\ell}} - q_{\ell}^{1/C_{\ell}}\leq \frac{1}{C_{\ell}} q_{\ell}^{1/C_{\ell}-1} \leq 2^{1-\ell}/c_1. 
\] 
By taking a sufficiently large $\ell\in \mathcal{J}$ such that $2^{1-\ell}/c_1$ is strictly less than the length of $U$, we obtain $A'\in U\subseteq \mathbb{R}\setminus \mathcal{W}(c_k)$. This is a contradiction. Therefore, we conclude that \eqref{Lemma-main-1} holds.

We next discuss the case when $A=\inf U$, which means that $A\in \subpartial_R\mathcal{W}(c_k)$. We assume that there exists an infinite set $\mathcal{J}\subseteq \mathcal{I}$ such that $p_{k+1}<(p_{k}+1)^{c_{k+1}}-p_{k}^{\theta c_{k+1}}$ for all $k\in \mathcal{J}$. Take a large $\ell \in \mathcal{J}$. Let $x=(p_{\ell}+1)^{c_{\ell+1}}-(p_{\ell}+1)^{\theta c_{\ell+1}}$ and $y=(p_{\ell}+1)^{c_{\ell+1}}-1$. Then it follows that $y-x=(p_{\ell}+1)^{\theta c_{\ell+1}}-1 \geq x^{\theta }$ and  $(p_{\ell}+1)^{\theta c_{\ell+1}}-1 \leq 2 x^\theta$ for $\ell$ sufficiently large. Thus,
\begin{equation}\label{Inequalities-y-x}
x^\theta \leq y-x \leq 2x^\theta.
\end{equation}
By Theorem~\ref{Lemma-BakerHarmanPintz} and \eqref{Inequalities-y-x}, there exists a constant $d_0'>0$ such that 
\[
\#([x,y]\cap \mathcal{P})\geq \#([x,x+x^\theta]\cap \mathcal{P}) \geq d_0 g(x^\theta) \geq d_0' g(y-x).
\]
We now apply Lemma~\ref{Lemma-Matomaki3} with $(e_2,e_3,\ldots):=(c_{\ell+2},c_{\ell+3},\ldots)$, $X:=x$, $Y:=y$, and $d_3:=d_0'$. Here $Y-X\in [X^{21/40},2X^{21/40}] \subset [X^{1/2}, ((3/2)^B-1)X ]$ since $\ell$ is sufficiently large. Therefore, there exists a sequence $(q_k)_{k={\ell+1}}^\infty$ of prime numbers such that $q_{\ell+1} \in [x,y]\cap \mathcal{P}$ and
for all integers $k\geq \ell+1$, we have $q_{k}^{c_{k+1}} \leq  q_{k+1} <(q_{k}+1)^{c_{k+1}} -1$. Let $q_k=p_k$ for all $k\in[\ell]$.
Similarly to the case when $A=\sup U$, we obtain
\[
q_{1}^{1/C_{1}} \leq  q_{2}^{1/C_{2}} \leq \cdots < (q_{2}+1)^{1/C_{2}}<(q_{1}+1)^{1/C_{1}}.
\]
Therefore, $A' :=\lim_{k\rightarrow \infty} q_{k}^{1/C_k}$ exists. Thus, $A'\in \mathcal{W}(c_k)$.  Recall that  $U\subseteq \mathbb{R}\setminus  \mathcal{W}(c_k)$. Similarly to the case when $A=\sup U$, it follows that $A'\in U$ by taking a sufficiently large $\ell$. This is a contradiction. Hence, we obtain Lemma~\ref{Lemma-key}.
\end{proof}

\begin{lemma}\label{Lemma-ConvergenceSpeed} Let $(c_k)_{k\in \mathbb{N}}$ be a sequence of real numbers satisfying all the conditions in Lemma~\ref{Lemma-key}. Fix any $A\in \subpartial \mathcal{W}(c_k)$. Let $p_k=\lfloor A^{C_k}\rfloor$ for every $k\in \mathbb{N}$. Let $\mathcal{I}$ and $k_0=k_0(A)$ be as in Lemma~\ref{Lemma-key}. Then for all $k\in \mathcal{I}\cap [k_0,\infty)$, 
\begin{enumerate}
\item if $A\in \subpartial_L \mathcal{W}(c_k)$, then  
$0\leq A-p_k^{1/C_k} \leq A\cdot p_1^{(\theta-1)C_{k+1}/c_1}$;
\item if $A\in \subpartial_R \mathcal{W}(c_k)$, then
$0\leq (p+1)^{1/C_k}- A \leq (p_1+1)^{1/c_1}A^{(\theta-1)C_{k+1}}$.
\end{enumerate}
\end{lemma}

\begin{proof}
Fix any $A\in \mathcal{W}(c_k)$. Take any $k\in \mathcal{I}\cap [k_0,\infty)$. Let us firstly discuss the case when $A\in \subpartial_L\mathcal{W}(c_k)$. By Lemma~\ref{Lemma-key},  we obtain $p_k^{c_{k+1}} \leq p_{k+1} \leq p_k^{c_{k+1}} +p_k^{\theta c_{k+1}}$. Therefore, by the definition of $A$ and $c_{k+1}\in \mathbb{N}$, 
\begin{equation}\label{Inequality-I1-Lemma24}
p_k^{c_{k+1}} \leq  A^{C_{k+1}} \leq p_k^{c_{k+1}}+p_k^{\theta c_{k+1}}+1\leq p_k^{c_{k+1}}+2p_k^{\theta c_{k+1}}.
\end{equation}
By the mean value theorem and \eqref{Inequality-I1-Lemma24}, there exists $\eta\in (0,2p_k^{(\theta-1)c_{k+1}})$ such that
\begin{align*}
0&\leq  A-p_k^{1/C_k} \leq  (p_k^{c_{k+1}}+2p_k^{\theta c_{k+1}})^{1/C_{k+1}}-p_k^{1/C_k} = p_k^{1/C_k} \left((1+2p_k^{(\theta-1) c_{k+1}})^{1/C_{k+1}}-1 \right) \\
&\leq A \cdot 2p_k^{(\theta-1)c_{k+1}} \frac{1}{C_{k+1}}\eta^{1/C_{k+1}-1} \leq A\cdot p_k^{(\theta-1)c_{k+1}} .
\end{align*}
By Lemma~\ref{Lemma-Mills}, we observe that $p_1^{C_{k}/c_1}\leq p_2^{C_{k}/(c_1c_2)}\leq \cdots \leq p_{k}$. Therefore, we conclude
\[
0\leq  A-p_k^{1/C_k} \leq A\cdot p_1^{(\theta-1) C_{k+1}/c_1}.
\]

Let us secondly discuss the case when $A\in \subpartial_R\mathcal{W}(c_k)$. By Lemma~\ref{Lemma-key}, we obtain $(p_{k}+1)^{c_{k+1}}-(p_{k}+1)^{\theta c_{k+1}}  \leq p_{k+1} < (p_{k}+1)^{c_{k+1}}$. Therefore, by the definition and $c_{k+1}\in \mathbb{N}$, we have
\begin{equation}\label{Inequality-I2-Lemma24}
(p_{k}+1)^{c_{k+1}}-(p_{k}+1)^{\theta c_{k+1}}  \leq A^{C_{k+1}} \leq (p_{k}+1)^{c_{k+1}}.
\end{equation}
By the mean value theorem and \eqref{Inequality-I2-Lemma24}, there exists $\eta'\in (0,(p_k+1)^{(\theta-1)c_{k+1}})$ such that
\begin{align*}
0&\leq (p_{k}+1)^{1/C_k} -A\leq (p_{k}+1)^{1/C_k} - ((p_{k}+1)^{c_{k+1}}-(p_{k}+1)^{\theta c_{k+1}}) ^{1/C_{k+1}}\\
&\leq (p_k+1)^{1/C_k} \cdot \left(1- (1- (p_k+1)^{(\theta-1)c_{k+1} })^{1/C_{k+1}}  \right) \\ 
&= (p_k+1)^{1/C_k}\cdot \frac{1}{C_{k+1}} (p_k+1)^{(\theta-1)c_{k+1}} (1-\eta')^{1/C_{k+1}-1} \\
&\leq (p_k+1)^{1/C_k} \cdot A^{(\theta-1)C_{k+1}}.
\end{align*}
By Lemma~\ref{Lemma-Mills}, we observe that $(p_1+1)^{1/C_1} \geq (p_2+1)^{1/C_2}\geq \cdots \geq (p_k+1)^{1/C_k}$. Therefore, 
\[
0\leq  (p_{k}+1)^{1/C_k} -A \leq (p_1+1)^{1/c_1} \cdot A^{(\theta-1)C_{k+1}}.
\]
\end{proof}

\section{Lemmas for evaluating lower bounds for $S$}\label{Section-LowerBounds}
In this section, we provide lemmas to evaluate lower bounds for $S$ in \eqref{Inequality-P(A1,A2,...,Ar)}. For all $f(x)=\sum_{j=0}^d a_j x^j\in \mathbb{Q}[x]$, we say that $L(f)=\sum_{j=0}^d |a_j|$ is the length of $f$. If $f(x)=a_d \prod_{j=1}^d (x-\alpha_j)$ for some $\alpha_1,\ldots ,\alpha_d\in \mathbb{C}$, then we define
\[
M(f) = |a_d| \prod_{j=1}^d \max(1, |\alpha_j|), 
\]
which is called the Mahler measure of $f$. By expanding $f(x)=a_d \prod_{j=1}^d (x-\alpha_j)$, 
\[
f(x)= a_d\sum_{N=0}^d (-1)^N \left(\sum_{1\leq j_1<\cdots <j_N\leq d } \alpha_{j_1}\cdots \alpha_{j_N}\right) x^{d-N}.     
\]
Therefore, we have
\begin{equation}\label{Inequality-LM}
L(f) \leq |a_d|  \sum_{N=0}^d \left| \sum_{1\leq j_1<\cdots <j_N\leq d } \alpha_{j_1}\cdots \alpha_{j_N} \right|      
\leq M(f) \sum_{N=0}^d \binom{d}{N} =2^d M(f).
\end{equation}
We refer the reader to the book written by Everest and Ward \cite[Chapter~1]{EverestWard} for more details on the Mahler measure. 

\begin{lemma}\label{Lemma-Independence}
Let $r\in \mathbb{N}$ and $n_1,\ldots , n_r\in \mathbb{N}$. Let $\alpha_1,\ldots, \alpha_r\in \mathbb{R}$ be algebraic numbers satisfying 
$\alpha_1^{n_1},\ldots , \alpha_r^{n_r} \in \mathbb{Q}$.
Assume that for all non-negative integers $u_1, \ldots ,u_r$  
\begin{equation}
\alpha_1^{u_1} \cdots  \alpha_r^{u_r} \notin \mathbb{Q},
\end{equation}
except for $u_1 \equiv 0 \mod n_1$, $u_2 \equiv 0 \mod n_2$, $\cdots$, $u_r \equiv 0 \mod n_r$. Then for all non-zero polynomials $P\in \mathbb{Q}[x_1,\ldots, x_r]$ with degree in $x_1,\ldots, x_r$ strictly less than $n_1,\ldots, n_r$, respectively,  we have $P(\alpha_1,\ldots,\alpha_r)\neq 0. $ 
\end{lemma} 
    \begin{proof}
    See \cite[Theorem~2]{Besicovitch} or \cite[THEOREM.~A]{mordell1953linear}.
    \end{proof}
For all algebraic numbers $\alpha$ we define $\ord(\alpha)$ as the minimum of positive integers $n$ such that $\alpha^n\in \mathbb{Q}$, and we define $\ord(\alpha)=\infty$ if no such $n$ exists. By the elementary theory of groups, for all $H\in \mathbb{N}$, \begin{equation}\label{Equation-order} 
H\equiv 0 \mod \ord(\alpha) 
\end{equation}
 if and only if $\alpha^H\in \mathbb{Q}$. The details can be seen in \cite{KobayashiSaitoTakeda}. 

\begin{lemma}\label{Lemma-order}
Let $D\geq 2$ be an integer. Then for all $p\in \mathcal{P}$, there exists a prime factor $q$ of $D$ such that $D/q\leq \ord ((p+1)^{1/D}) \leq D.$
\end{lemma}
 
\begin{proof}
 Fix any $p\in \mathcal{P}$. Let $\alpha=(p+1)^{1/D}$. Let $H= \ord (\alpha)$. It is trivial that $H\leq D$ since $\alpha^D\in \mathbb{Q}$. Let us show that $D/q\leq H$ for some prime factor $q$ of $D$. Since $(p+1)^{H/D}\in \mathbb{Q}$,  there exists $b\in \mathbb{N}$ such that $(p+1)^H=b^{D}$. Let $d=\gcd(H,D)$, $H'=H/d\in \mathbb{N}$, and $q=D/d\in \mathbb{N}$. Then we have $(p+1)=b^{q/H'}$, which implies that $b=b'^{H'}$ for some $b'\in \mathbb{N}$. Therefore, $p=b'^{q}-1$. This yields $q\in \mathcal{P}$ since $p\in \mathcal{P}$. Hence, $H\geq d =D/q$.    
\end{proof}

\begin{lemma}\label{Lemma-degree}
Let $D\geq 2$ be an integer. For all $p_1,\ldots, p_r \in \mathcal{P}$ with $p_1+1<p_2<\cdots <p_r$, we have the following: 
\begin{enumerate}
    \item \label{Condition-leftcase} $[\mathbb{Q}(p_1^{1/D}) \colon \mathbb{Q}]=D$ and for every $2\leq j\leq r$, 
    \[
    [\mathbb{Q}(p_1^{1/D},\ldots, p_{j}^{1/D})\colon \mathbb{Q}(p_1^{1/D},\ldots, p_{j-1}^{1/D }) ]=D; 
    \]
    \item \label{Condition-rightcase} $[\mathbb{Q}((p_1+1)^{1/D}) \colon \mathbb{Q}]=\ord((p_1+1)^{1/D}) $ and for every $2\leq j\leq r$, 
    \[
    [\mathbb{Q}((p_1+1)^{1/D},p_2^{1/D}, \ldots, p_{j}^{1/D })\colon \mathbb{Q}((p_1+1)^{1/D},p_2^{1/D}, \ldots, p_{j-1}^{1/D }) ]=D. 
    \]
\end{enumerate}
\end{lemma}
\begin{proof}Let $q_1=p_1$ or $p_1+1$, $\alpha_1=q_1^{1/D}$, and let  $\alpha_j=p_j^{1/D}$ for every $2\leq j\leq r$. Let $H_j=\ord(\alpha_j)$ for all $j\in [r]$. We observe that each $\alpha_j$ is a root of the polynomial $x^{H_j} -\alpha_j^{H_j}\in \mathbb{Q}[x]$. Therefore, since $\ord (\alpha_j)=D$ for all $2\leq j \leq r$, it suffices to show that for all non-zero polynomials $P\in \mathbb{Q}[x_1,\ldots,x_r]$ with degree in each $x_j$ strictly less than $\ord(\alpha_j)$, we have $P(\alpha_1,\ldots, \alpha_r)\neq 0$. Suppose that $\alpha_1^{u_1} \cdots \alpha_r^{u_r}\in \mathbb{Q}$ for some non-negative integers $u_1,\ldots, u_r$. Then there exists $b\in \mathbb{N}$ such that $q_1^{u_1} p_2^{u_2} \cdots p_r^{u_r}=b^D$. Since $q_1<p_2<\cdots <p_r$ and  
$p_2,\ldots ,p_r$ are distinct prime numbers, $u_2 \equiv 0, \ldots, u_r\equiv 0 \mod D$.  Hence, $q_1^{u_1}\in \mathbb{Q}$. By \eqref{Equation-order}, we have $u_1\equiv 0 \mod \ord(\alpha_1)$. Therefore, by Lemma~\ref{Lemma-Independence}, we have $P(\alpha_1,\ldots, \alpha_r)\neq 0$.  
\end{proof}

\section{Proof of Theorem~\ref{Theorem-main1}}\label{Section-proof} 

In this section, let $r$ be a positive integer and let $(c_k)_{k\in\mathbb{N}}$ be a sequence of integers satisfying the three conditions in Theorem~\ref{Theorem-main1}. Fix any  $A_1\in \subpartial \mathcal{W}(c_k)$ and $A_2,\ldots, A_{r}\in \subpartial_L \mathcal{W}(c_k)$ with $A_1<A_2<\cdots <A_r$. Let $P(x_1,\ldots,x_r)\in \mathbb{Z}
[x_1,\ldots, x_r]$ be any non-zero polynomial. Let $d$ be a positive integer such that the degree of $P$ in each $x_i$ is less than or equal to $d$. Let $p_j(k)=\lfloor A_j ^{C_k} \rfloor$ for all $1\leq j\leq r$ and $k\in \mathbb{N}$. By Proposition~\ref{Proposition-closed}, $\mathcal{W}(c_k)$ is closed. Thus, $A_1,\ldots, A_r \in \subpartial\mathcal{W}(c_k)\subseteq \partial \mathcal{W}(c_k)\subseteq \mathcal{W}(c_k)$. This implies that $p_j(k)\in \mathcal{P}$ for all $1\leq j\leq r$ and $k\in \mathbb{N}$. By $A_1<\cdots <A_r$, there exists $k_1>0$ such that for all $k\geq k_1$ we have 
\begin{equation}\label{Inequality-increasing-pj}
p_1(k)+1<p_2(k)<\cdots <p_r(k). 
\end{equation}
Let 
\[
\alpha_1=\alpha_1(k)=
\begin{cases}
(p_1(k))^{1/C_k}&\quad \text{if $A_1\in \subpartial_L \mathcal{W}(c_k)$},\\
(p_1(k)+1)^{1/C_k}&\quad \text{if $A_1\in \subpartial_R \mathcal{W}(c_k)$},  
\end{cases}
\]
and $\alpha_j=\alpha_j(k)=p_j(k)^{1/C_k}$ for all $2\leq j \leq r$. For all $Q(x_1,\ldots, x_r)\in \mathbb{R}[x_1,\ldots, x_r]$, if  
$\displaystyle{
Q(x_1,\ldots, x_r)=\sum_{j_1,{\ldots},j_r}  a_{j_1,\ldots, j_r} x_1^{j_1} \cdots x_r^{j_r}}$, then we define 
\[
|Q|(x_1,\ldots , x_r) = \sum_{j_1,\ldots ,j_r} |a_{j_1,\ldots ,j_r}|x_1^{j_1} \cdots x_r^{j_r}.
\]

\begin{lemma}\label{Lemma-quantitative}For all $k\in\mathbb{N}$ with $k\geq k_1$ and $2^{k-2} >d$, we have 
\begin{equation}\label{Inequality-quantitative}
\left(4+4 |P| (\alpha_1(k),\cdots, \alpha_r(k)) \right)^{-C_k^r}\leq  |P(\alpha_1(k),\ldots, \alpha_r(k)) | . 
\end{equation}
\end{lemma}

\begin{proof}
Let $\gamma=P(\alpha_1,\ldots, \alpha_r)$, which is an algebraic integer. It is clear that the left-hand side of \eqref{Inequality-quantitative} is less than $1$. Hence, we may assume that $|\gamma|<1$.  Let $K= \mathbb{Q}(\alpha_1,\ldots, \alpha_r)$. Let $G(K)$ be the set of all complex embeddings $\sigma$ from $K$ to $\mathbb{C}$,  that is, $\sigma$ sends $\alpha_j$ to their conjugates and $\sigma(a)=a$ for all $a\in \mathbb{Q}$. 
Define
\[
\varphi(x)= \prod_{\sigma\in G(K)} (x- \sigma (\gamma)  ).  
\]
Here $\varphi(0)$ is equal to the norm of the algebraic integer $\gamma$, up to the signs. Therefore, $\varphi(0)\in \mathbb{Z}$. Let us show that $\varphi(0)\neq 0$. Here each $\sigma\in G(K)$ is injective. Thus, if $\sigma(\gamma)=0$ for some $\sigma\in G(K)$, then $\gamma=0$. Therefore, it suffices to show that $\gamma\neq 0$. Here by Lemma~\ref{Lemma-order} with $p:=p_1(k)$, we have $\ord(\alpha_1)\geq C_k/q$ for some prime factor $q$ of $C_k$. Therefore, $\ord(\alpha_1)\geq 2^{k-2}>d$. By the choice of $d$ and Lemma~\ref{Lemma-degree} with $p_1:=p_1(k), \ldots, p_r:=p_r(k)$, we have
\[
[\mathbb{Q}(\alpha_1): \mathbb{Q}]>d,\quad  [\mathbb{Q}(\alpha_1,\ldots ,\alpha_{j})\colon \mathbb{Q}(\alpha_1,\ldots ,\alpha_{j-1})] =C_k\geq  2^{k-1} >d \quad (2\leq j\leq r).
\]
This implies that $\gamma\neq 0$.

Since $\varphi(\gamma)=0$ and $\varphi(0)\in \mathbb{Z}\setminus\{0\}$, there exists $\eta\in (-|\gamma|, |\gamma|)$ such that
\[
1\leq |\varphi(0)|=|\varphi(0)-\varphi(\gamma)| = |\gamma| |\varphi'(\eta) |.
\] 
by the mean value theorem. Further, by $|\eta|<|\gamma|<1$, we obtain 
\[
|\varphi'(\eta)| \leq (\deg \varphi)  L(\varphi).
\]
By \eqref{Inequality-LM},  we obtain $L(\varphi)\leq 2^{\deg{\varphi} }  M(\varphi) $. Therefore,
\[
1\leq |\gamma| |\varphi'(\eta) |\leq (\deg{\varphi}) |\gamma| \cdot 2^{\deg{\varphi} }M(\varphi)\leq 4^{\deg{\varphi} }|\gamma| M(\varphi),
\]
where we use the inequality $x\leq 2^x$ for all $x\geq 1$. By the definition of the Mahler measure of $\varphi$, 
\[
M(\varphi)= \prod_{\beta\colon \varphi(\beta)=0} \max(1,|\beta|)\leq \prod_{\sigma\in G(K) } (1+ |P(\sigma(\alpha_1),\ldots, \sigma(\alpha_r))  |).
\]
Further, since $\alpha_j^{C_k}\in \mathbb{Q}$, for all $\sigma\in G(K)$ and $j\in [r]$, we have
\[
\sigma(\alpha_j)^{C_k} - \alpha_j^{C_k}= \sigma\left(\alpha_j^{C_k} - \alpha_j^{C_k}    \right)=0. 
\]
 This implies that  $|\sigma(\alpha_j)|=\alpha_j$. Hence,
\[
|P(\sigma(\alpha_1),\ldots, \sigma(\alpha_r))  | \leq |P|(|\sigma(\alpha_1)|,\ldots,|\sigma(\alpha_r)|)= |P| (\alpha_1,\cdots, \alpha_r).
\]
By the theory of fields and Lemma~\ref{Lemma-degree}, we see that  
\[
\deg \varphi =\#  G(K) = [\mathbb{Q}(\alpha_1,\ldots, \alpha_r)\colon \mathbb{Q}] \leq C_k^r.
\]
Therefore, we conclude that  $1\leq  |\gamma| \cdot (4+4 |P| (\alpha_1,\cdots, \alpha_r) )^{C_k^r}. $
\end{proof}

\begin{proof}[Proof of Theorem~\ref{Theorem-main1}]Let $\mathcal{I}, k_0$ be as in Lemma~\ref{Lemma-key}. If necessary, we replace 
\[
\max(k_0(A_1),\ldots, k_0(A_r), k_1, \lfloor \log(4d)/\log2\rfloor +1 )
\]
with $k_1$. Take any large $k\in \mathcal{I}\cap[k_1, \infty)$. By the triangle inequality,  
\begin{align*}
|P(A_1,\ldots, A_r) |\geq |P(\alpha_1,\ldots ,\alpha_r)|-|P(A_1,\ldots, A_r)-P(\alpha_1,\ldots, \alpha_r)|  =: S-T. 
\end{align*}
Lemma~\ref{Lemma-quantitative} implies that $S\geq (4+4 |P|(\alpha_1,\ldots, \alpha_r ))^{-C_k^r}$. Further, it follows that 
\[
 \alpha_1\leq (p_1+1)^{1/C_1}\leq 2p_1^{1/C_1}\leq 2A_1, \quad  \alpha_j \leq A_j\quad (2\leq j\leq r).
\]
Letting $\gamma_1=4+4 |P|(2A_1,A_2,\ldots, A_r )>1$, this yields that 
\begin{equation}\label{Inequality-S}
S \geq \gamma_1^{-C_k^r}.
\end{equation}
 By the triangle inequality and the mean value theorem, there exist $\eta_1\in [\alpha_1, A_1]\cup[A_1,\alpha_1] ,\eta_2\in[\alpha_2, A_2],  \ldots ,\eta_r\in[\alpha_r, A_r] $ such that 
\begin{align} \nonumber
&T\leq |P(A_1,\ldots, A_r)- P(\alpha_1, A_2,\ldots ,A_r)| + |P(\alpha_1, A_2,\ldots ,A_r)-P(\alpha_1, \alpha_2, A_3, \ldots ,A_r)|\\ \nonumber
&\quad\quad\quad+\cdots +|P(\alpha_1, \ldots ,\alpha_{r-1}, A_r)- P(\alpha_1,\ldots, \alpha_r)|\\ \label{Inequality-upperbound1}
&\leq |A_1-\alpha_1| \left|\frac{\partial}{\partial x_1} P(x_1, A_2,\ldots ,A_r)|_{x_1=\eta_1} \right| +|A_2-\alpha_2| \left|\frac{\partial}{\partial x_2} P(\alpha_1, x_2,\ldots ,A_r)|_{x_2=\eta_2} \right| \\ \nonumber
&\quad\quad +\cdots + |A_r-\alpha_r| \left|\frac{\partial}{\partial x_r} P(\alpha_1, \ldots ,\alpha_{r-1}, x_r)|_{x_r=\eta_r} \right|. 
\end{align}
 By $\max(\alpha_1, \eta_1)\leq 2A_1$ and $\eta_j\leq A_j$ $(2\leq j\leq r)$, for all $j \in [\ell]$ we have
\begin{equation}\label{Inequality-upperbound2}
\left|\frac{\partial}{\partial x_j} P(\alpha_1,\ldots, \alpha_{j-1},x_j, A_{j+1},\ldots  A_r)|_{x_j=\eta_j} \right| \leq  \frac{\partial}{\partial x_j} |P|(2A_1, A_2,\ldots,A_r).
\end{equation}
 Hence by Lemma~\ref{Lemma-ConvergenceSpeed}, \eqref{Inequality-upperbound1}, and \eqref{Inequality-upperbound2},  there exist $F=F(P,A_1,\ldots,A_r,c_1)>0$ and $\gamma_2=\gamma_2(A_1,\ldots, A_r,c_1)>1$ such that 
\begin{equation}\label{Inequality-T}
T\leq F \gamma_2^{C_{k+1}(\theta-1)}.
\end{equation}
Hence by combining \eqref{Inequality-S} and \eqref{Inequality-T},  we have
\[
|P(A_1,\ldots, A_r)| \geq \gamma_1^{-C_k^r} -F\gamma_2^{C_{k+1}(\theta-1)}.
\]
Here we observe that $\gamma_1^{-C_k^r}\geq F\gamma_2^{C_{k+1}(\theta-1)}/2$ for some large $k$. Indeed, 
\begin{align*}
& \gamma_1^{-C_k^r}\geq F\gamma_2^{C_{k+1}(\theta-1)}/2 \Leftrightarrow (2/F) \left(\gamma_1^{-1} \gamma_2^{c_{k+1}C_k^{1-r}(1-\theta) }\right)^{C_k^r}\geq 1.
\end{align*}
By \eqref{Condition-Ck} in Theorem~\ref{Theorem-main1}, there exists $k\in \mathcal{I}\cap[k_1,\infty)$ satisfying the last inequality. Therefore, there exists a large $k\in \mathbb{N}$ such that
\[
|P(A_1,\ldots, A_r) | \geq \gamma_1^{-C_k^r}/2>0,
\]
which means that $\{A_1, \ldots ,A_r\}$ is algebraically independent.
\end{proof}

\section{Further discussions}\label{Section-Further}

\subsection{Rational approximation}
By the idea of the proof of Theorem~\ref{Theorem-main1}, we find a quantitative rational approximation of elements in $\subpartial\mathcal{W}(c_k)$ if $(c_k)_{k\in \mathbb{N}}$ is concretely given.
\begin{proposition}
For all $A\in \subpartial \mathcal{W}(k)$, there exists a constant $\eta=\eta(A)>0$ such that for all integers $m,n\geq 3$, we have
\[
|A-m/n| \geq \exp\left(-\exp \left(\eta \log M \log\log M  \right) \right),
\]
where $M=\max(m,n)$. 
\end{proposition}
\begin{proof}
By substituting $c_k:=k$, $(c_k)_{k\in \mathbb{N}}$ satisfies the conditions in Theorem~\ref{Theorem-main1}. Fix any positive integers $m,n$, and fix any $A\in \subpartial \mathcal{W}(k)$. Let $p_k=\lfloor A^{C_k}\rfloor$ for all $k\in\mathbb{N}$, and define $\alpha=\alpha(k)=p_k^{1/C_k}$ if $A\in \subpartial_L \mathcal{W}(c_k)$ and $\alpha=\alpha(k)=(p_k+1)^{1/C_k}$ if $A\in \subpartial_R \mathcal{W}(c_k)$. Let $\mathcal{I}$ and $k_0=k_0(A)$ be as in Lemma~\ref{Lemma-key}. We see that $\mathcal{I}=\{2,3,\ldots\}$ since $c_{k+1}=k+1$. Take any $k\in \mathcal{I}\cap[k_0,\infty)$. Let $P(x)=nx-m$. Let $M=\max(m,n)$. Then by the proof of Theorem~\ref{Theorem-main1}, we have
\[
|P(A)| \geq (4+4\cdot (2A+1)M  )^{-C_k} - |P(A)-P(\alpha)|\geq(8(A+1)M  )^{-C_k} - |P(A)-P(\alpha)|
\]
and by Lemma~\ref{Lemma-ConvergenceSpeed}, there exist constants $\eta_0>0$ and $\gamma>1$ which depend only on $A$ such that
$
|P(A)-P(\alpha)|=n|A-\alpha |\leq \eta_0 M \gamma^{C_{k+1}(\theta-1)}.  
$
Thus, there exists $\eta_1=\eta_1(A)>0$ such that
\begin{align*}
&\eta_0M\gamma^{(k+1)!(\theta-1)} \leq  \frac{1}{2}(8 (A+1)M )^{-k!}\\
&\Leftrightarrow \log \eta_0 + \log M +(k+1)!(\theta-1)\log \gamma \leq \log(1/2)-k!\log (8(A+1)M)\\
&\Leftarrow k\geq \eta_1 \log M.
\end{align*}
Therefore, by substituting $k=\lfloor \eta_1\log M\rfloor+1$, there exist $\eta_2=\eta_2(A)>0$, $\eta_3=\eta_3(A)>0$, and $\eta=\eta(A)>0$ such that  
\begin{align*}
&|A-m/n|\geq \frac{1}{2M} (8(A+1)M)^{-k!} \geq \eta_2 M^{-k^k-1} \geq \eta_2 M^{- (\eta_3 \log M)^{\eta_3 \log M}  } \\
&\geq \exp \left(-\exp \left(\eta \log M\log\log M \right)\right). 
\end{align*}

\end{proof}

\subsection{Linear independence over $\mathbb{Q}$} We say that $\{\alpha_1,\ldots , \alpha_r\}\subseteq \mathbb{C}$ is linearly independent over $\mathbb{Q}$ if for all $(\lambda_1,\ldots, \lambda_r)\in \mathbb{Q}^r\setminus \{(0,\ldots ,0)\}$ we have $\lambda_1 \alpha_1 + \cdots + \lambda_r \alpha_r\neq 0$. In Theorem~\ref{Theorem-main1} with $r=2$, we assume 
\begin{equation}\label{Condition-Ck-r=2}
\limsup_{k\rightarrow \infty} c_{k+1}C_k^{-1}=\infty
\end{equation}
 to obtain the algebraic independence of $\{A_1,A_2\}$, where $A_1\in \subpartial \mathcal{W}(c_k)$ and $A_2\in \subpartial_L \mathcal{W}(c_k)$ with $A_1<A_2$. We can make condition \eqref{Condition-Ck-r=2} weaker when we discuss the linear independence over $\mathbb{Q}$ of $\{A_1, A_2\}$. 

\begin{proposition}\label{Proposition-linearly-ind}
Let $(c_k)_{k\in \mathbb{N}}$ be a sequence of integers satisfying that 
\begin{enumerate}
\item $c_1\geq 1$;
\item $c_k\in \mathbb{Z}$ and $c_k\geq 2$ for all integers $k\geq 2$;
\item \label{Condition-linInd}$\limsup_{k\rightarrow \infty} c_{k+1}=\infty$.
\end{enumerate}
 Then for all $A_1\in \subpartial \mathcal{W}(c_k)$ and $A_2\in \subpartial_L \mathcal{W}(c_k)$ with $A_1<A_2$, $\{A_1,A_2\}$ is linearly independent over $\mathbb{Q}$.
\end{proposition}

\begin{proof}
Assume that there were $m,n\in \mathbb{Z}\setminus \{0\}$ such that $mA_2-nA_1=0$. Let $p_j(k)=\lfloor A_j^{C_k}\rfloor$ for all $j\in \{1,2\}$ and $k\in \mathbb{N}$. Let $\mathcal{I}$ and $k_0$ be as in Lemma~\ref{Lemma-key}. Let $k_1=\max(k_0(A_1), k_0(A_2))$. Take any large $k\in \mathcal{I}\cap [k_1,\infty)$. We may assume that $p_1(k)+1<p_2(k)$ since $A_1<A_2$. We define 
\[
q_1(k)=
\begin{cases}
p_1(k) \quad \text{if $A_1\in \subpartial_L \mathcal{W}(c_k)$},\\
p_1(k)+1 \quad \text{if $A_1\in \subpartial_R \mathcal{W}(c_k)$}.
\end{cases}
\]
Let $\varphi(x)=x^{C_k}-q_1(k)/p_2(k)$, and let $\alpha_1(k)=q_1(k)^{1/C_k}$ and $\alpha_2(k)=p_2(k)^{1/C_k}$. Then 
\[
\varphi(m/n) = \frac{p_2(k)m^{C_k}-q_1(k)n^{C_k}}{n^{C_k}p_2(k) }.
\]
Let us discuss the case when $A_1\in \subpartial_L \mathcal{W}(c_k)$. Suppose that $\varphi(m/n)= 0$. Then 
\begin{equation}\label{Equation-f1-Further}
p_1(k)n^{C_k}=p_2(k)m^{C_k} .
\end{equation}
 The number of the prime factor $p_1(k)$ on the left-hand side of \eqref{Equation-f1-Further} is $1+C_ku_1$ for some non-negative integer $u_1$, but that on the right-hand side is $C_k u_2$ for some non-negative integer $u_2$. This is a contradiction. Therefore, $\varphi(m/n)\neq 0$. In the case when $A_1\in \subpartial_R \mathcal{W}(c_k)$, we similarly obtain $\varphi(m/n)\neq 0$. 
Therefore, $\varphi(m/n)\neq 0$ in both cases. Hence, we have $|\varphi(m/n)|\geq (n^{C_k}p_2(k))^{-1}$. Thus by the mean value theorem, we obtain 
\begin{align*}
&\frac{1}{n^{C_k}A_2^{C_k} }\leq \frac{1}{n^{C_k}p_2(k)} \leq |\varphi(m/n)|=|\varphi(A_1/A_2)|= \frac{1}{A_2^{C_k}p_2(k)} |q_1(k)A_2^{C_k}- p_2(k)A_1^{C_k}|\\
&\leq \frac{1}{A_2^{C_k}p_2(k)} (A_2^{C_k}|q_1(k)-A_1^{C_k}| +A_1^{C_k}|p_2(k) -A_2^{C_k}|) \\
&\leq \frac{1}{A_2^{C_k}p_2(k) } \left(A_2^{C_k}|\alpha_1(k)-A_1| C_k (A_1+1)^{C_k-1}+ A_1^{C_k}|p_2(k) -A_2|C_k(A_2+1)^{C_k-1} \right).
\end{align*}
Here we observe that $p_2(1)^{1/C_1}\leq p_2(2)^{1/C_2}\leq \cdots \leq p_2(k)^{1/C_k} $ by Lemma~\ref{Lemma-Mills}. Therefore, there exists constant $\gamma_1=\gamma_1(A_1,A_2,n,c_1)>1$  such that
\[
1\leq \gamma_1^{C_k} \left( |\alpha_1(k)-A_1| + |\alpha_2(k) -A_2|\right).
\] 
By Lemma~\ref{Lemma-ConvergenceSpeed}, there exist constants $\eta'=\eta(A, c_1)>0$ and $\gamma_2=\gamma_2(A,c_1)>1$ such that 
\begin{equation}\label{Inequality-linear}
1\leq \gamma_1^{C_k}  (|\alpha_1(k)-A_1| + |\alpha_2(k) -A_2|)\leq \eta' \gamma_1^{C_k} \gamma_2^{C_{k+1}(\theta-1)}   =\eta' (\gamma_1 \gamma_2^{c_{k+1}(\theta-1)})^{C_k}.
\end{equation}
By \eqref{Condition-linInd} in Proposition~\ref{Proposition-linearly-ind}, there exists $k\in \mathcal{I}\cap [k_1,\infty)$ such that the far right-hand side of \eqref{Inequality-linear} is strictly less than $1$, which is a contradiction. Therefore, $\{A_1,A_2\}$ is linearly independent over $\mathbb{Q}$.
\end{proof}

By substituting $c_k:=k$ in Proposition~\ref{Proposition-linearly-ind}, we conclude the following corollary. 

\begin{corollary}
For all $A_1\in \subpartial \mathcal{W}(k)$ and $A_2\in \subpartial_L \mathcal{W}(k)$ with $A_1<A_2$, $\{A_1,A_2\}$ is linearly independent over $\mathbb{Q}$.
\end{corollary}

\subsection{Numerical calculations}Some studies have conducted numerical calculations of PRCs. As we mentioned in Section~\ref{Section-introduction}, Caldwell and Cheng computed 600 digits of the decimal expansion of Mills' constant under the Riemann hypothesis \cite{CaldwellCheng}. Elsholtz studied an unconditional method for the calculation of PRCs. For sufficiently large $c>1$, he gave an unconditional method for the computation of some $A\in \mathcal{W}(c)$ \cite[Theorem~ a)]{Elsholtz}.  As an application, some $A\in \mathcal{W}(10^{10})$ can be computed to millions of decimal places \cite[Theorem~b)]{Elsholtz}. In this section, let us show the calculation of the minimum of $\mathcal{W}(c_k)$ for several sequences $(c_k)_{k\in \mathbb{N}}$ by using \textit{Mathematica}. The notebook file which we used can be seen in \cite{STnote}.

\begin{proposition}\label{Proposition-3k!}
The decimal expansion of the minimum of 
$
\{A>1\colon \lfloor A^{3^{k!}}\rfloor \text{ is a PRF} \}
$
begins as $1.3052998807\cdots$ and can be computed to $86$ decimal places.
\end{proposition}
Note that this result is unconditional. Specifically, we do not suppose any hypotheses on prime numbers. In addition, this constant is transcendental by Corollary~\ref{Corollary-W(3k!)}.
\begin{proposition}\label{Proposition-k!}
Assuming the Riemann hypothesis, the decimal expansion of the minimum of $\mathcal{W}(k)$ begins as $2.2419914653\cdots$ and can be computed to $254$ decimal places.  
\end{proposition}

 We apply the following two theorems.
\begin{theorem}[{\cite[p.36, Chapter~2]{Mattner}}]\label{Theorem-Mattner}
Let $m\geq 1438989$ be an integer. For all $n\in \mathbb{N}$, there exists a prime number $p$ such that $n^m< p<(n+1)^m$.  
\end{theorem}

\begin{remark}
Recently, Cully-Hugill proved that Theorem~\ref{Theorem-Mattner} is still true even if we replace $m\geq 1438989$ with $m\geq 180$ \cite[Theorem 1]{cu2021}. By this result, we can compute that the decimal expansion of the minimum of 
$
\{A>1\colon \lfloor A^{2^{k!}}\rfloor \text{ is a PRF} \}
$
begins as $1.49534878122\cdots$. It is very tough to calculate this minimum if we only apply Theorem~\ref{Theorem-Mattner}. 

\end{remark}
 Carneiro, Milinovich, and Soundararajan \cite{CMS} showed that if we assume the Riemann hypothesis, for $x\ge4$ there exists a prime number $p\in (x,x+\frac{22}{25}\sqrt x\log x)$.
In particular, for $n^m\ge4$, there exists a prime number 
\[
p\in \left(n^m,n^m+\frac{22m}{25}n^{m/2}\log n\right)\subset (n^m,(n+1)^m).
\]
If $n=1$, then for $m\ge 3$, $2\in (1,2^{m})$. Therefore, the following theorem holds.
\begin{theorem}[cf. {\cite[p.538, Theorem 1.5]{CMS}}]\label{Theorem-CMS}
Assume the Riemann hypothesis. Let $m\geq 3$. For all $n\in \mathbb{N}$, there exists a prime number $p$ such that $n^m < p <(n+1)^m$. 
\end{theorem}
Note that $p= (n+1)^m-1$ is not possible if $n\geq 2$ since $p$ needs to be a prime number. Therefore, the upper bounds $(n+1)^m$ in Theorem~\ref{Theorem-Mattner} and Theorem~\ref{Theorem-CMS} can be replaced with $(n+1)^m-1$ for every integer $n\geq 2$.

\begin{lemma}\label{Lemma-way-of-construction}
Let $(c_k)_{k\in \mathbb{N}}$ be a sequence of real numbers satisfying that 
\begin{enumerate}
    \item \label{Condition-c1-const} $c_1>0$;
    \item \label{Condition-ck+1-const} $c_{k+1} \geq 2$ and $c_{k+1}\in \mathbb{Z}$ for all $k\in \mathbb{N}$.
\end{enumerate}
Suppose that there exists a sequence of prime numbers $(p_k)_{k\in \mathbb{N}}$ such that 
\begin{enumerate}\setcounter{enumi}{2} 
\item \label{Condition-Mills}$ p_k^{c_{k+1}}\leq p_{k+1} <(p_k+1)^{c_{k+1}}-1$ for all $k\in \mathbb{N}$;
\item \label{Condition-minimum}there exists $k_0\in \mathbb{N}$ such that for all $k\geq k_0$
\[
p_{k+1}=\min \left([p_k^{c_{k+1}},(p_k+1)^{c_{k+1}}-1) \cap \mathcal{P}\right).
\]
\end{enumerate}
Then $A=\lim_{k\rightarrow \infty} p_k^{1/C_k}$ exists, and $A\in \partial_L \mathcal{W}(c_k)$. In addition, if $p_1=2$ and we can take $k_0=1$ in \eqref{Condition-minimum}, then we have $A=\min \mathcal{W}(c_k)$. 
\end{lemma}

\begin{proof}
Let $(p_k)_{k\in \mathbb{N}}$ be a sequence of prime numbers satisfying \eqref{Condition-Mills} and \eqref{Condition-minimum}. Then by \eqref{Condition-Mills}, $A=\lim_{k\rightarrow \infty} p_k^{1/C_k}$ exists, and $\lfloor A^{C_k}\rfloor =p_k$ holds for every $k\in \mathbb{N}$. Therefore $A\in \mathcal{W}(c_k)$. Let us next show that $A\in \subpartial_{L} \mathcal{W}(c_k)$. Let $k_0$ be an integer which satisfies \eqref{Condition-minimum}. Take any $k\geq k_0$. Then $p_{k}^{c_{k+1}}$ is not a prime number since $c_{k+1}\geq 2$ and $c_{k+1}\in \mathbb{Z}$. Therefore the interval $[p_{k}^{c_{k+1}}, p_{k+1})\cap \mathbb{N} $ is non-empty and has no prime numbers by \eqref{Condition-minimum}. Hence, no $B\in [p_{k}^{1/C_{k}}, p_{k+1}^{1/C_{k+1}})$ belongs to $\mathcal{W}(c_k)$ since $\lfloor B^{C_{k+1}} \rfloor \in [p_{k}^{c_{k+1}}, p_{k+1})\cap \mathbb{N}$ and the interval $[p_{k}^{c_{k+1}}, p_{k+1})\cap \mathbb{N}$ contains no prime numbers. Therefore, 
\[
[p_{k_0}^{1/C_{k_0}},A) =\bigcup_{k\geq k_0 } [p_{k}^{1/C_{k}}, p_{k+1}^{1/C_{k+1}}) \subseteq \mathbb{R}\setminus \mathcal{W}(c_k).  
\]
This implies that $A\in \subpartial_{L} \mathcal{W}(c_k)$ by the definition of the left sub-boundary. 

Furthermore, in the case when $p_1=2$ and $k_0=1$, let us show that $A=\min \mathcal{W}(c_k)$. Take any $B\in \mathcal{W}(c_k)$. Let $q_k= \lfloor B^{C_k}\rfloor$ for every $k\in \mathbb{N}$. By Lemma~\ref{Lemma-Mills}, we obtain 
\begin{equation}\label{Ineuqlity-qk-Further}
q_k^{c_{k+1}}\leq q_{k+1}<(q_k+1)^{c_{k+1}}-1
\end{equation}
for every $k\in \mathbb{N}$. Then $\lfloor A^{c_1} \rfloor =p_1=2 \leq q_1 =\lfloor B^{c_1}\rfloor$. Thus, we have either $p_1=q_1$ or $A<B$. If $p_1=q_1$ holds, then by \eqref{Condition-minimum} and \eqref{Ineuqlity-qk-Further}, it follows that $p_2\leq q_2$. Therefore, we have either $p_2=q_2$ or $A<B$. By iterating this argument, we obtain $p_k=q_k$ for every $k\in \mathbb{N}$ or $A<B$. In the former case, by the mean value theorem and $\min(A,B)>1$, for every $k\in \mathbb{N}$
\[
1>|A^{C_k}-B^{C_k}| \geq  {C_k} |A-B|.
\]
Therefore, $A=B$ by taking $k\rightarrow \infty$. Hence, $A=\min \mathcal{W}(c_k)$.  
\end{proof}

\begin{remark}
If we replace the symbol ``$\min$" with ``$\max$" in \eqref{Condition-minimum} in Lemma~\ref{Lemma-way-of-construction}, then $A=\lim_{k\rightarrow \infty} p_k^{1/C_k}$ also exists and $A\in \subpartial_R \mathcal{W} (c_k)$. This proof is similar to the proof of Lemma~\ref{Lemma-way-of-construction}, so we omit it. 
\end{remark}

\begin{proof}[Proof of Proposition~\ref{Proposition-3k!}]
Let $c_1=3$ and $c_k= 3^{k!-(k-1)!}$ for every $k\geq 2$. Then $(c_k)_{k\in \mathbb{N}}$ satisfies \eqref{Condition-c1-const} and \eqref{Condition-ck+1-const} in Lemma~\ref{Lemma-way-of-construction}. Let $p_1=2$. Since $p_1^{c_{2}}= 2^3=8$ and $ (p_1+1)^{c_{2}}-1= 3^3-1=26$, 
\[
\min\left([p_1^{c_{2}}, (p_1+1)^{c_{2}}-1 )\cap \mathcal{P}\right) =11.
\]
Let $p_2=11$. By using \texttt{ProvablePrimeQ[n]} in \textit{Mathematica}, we obtain
\[
\min\left([p_2^{c_{3}}, (p_2+1)^{c_{3}}-1 )\cap \mathcal{P}\right)
    =\min\left([11^{3^4}, 12^{3^4}-1 )\cap \mathcal{P}\right)=11^{3^4}+140.
\]
Let $p_3=11^{3^4}+140$. Here we observe that for all $k\geq 3$, 
\[
c_{k+1}=3^{(k+1)!-k!}= 3^{k\cdot k!} \geq 3^{3\cdot 3!} =3^{18}>1438989.
\]
Hence, $[p_3^{c_4}, (p_3+1)^{c_4}-1)\cap \mathcal{P}$ is non-empty by Theorem~\ref{Theorem-Mattner}. Therefore,  
\[
p_4= \min \left( [p_3^{c_4}, (p_3+1)^{c_4}-1)\cap \mathcal{P}\right)
\]
exists. By iterating this argument, we construct a sequence of prime numbers $(p_k)_{k\in \mathbb{N}}$ satisfying \eqref{Condition-Mills} and \eqref{Condition-minimum} with $k_0=1$ in Lemma~\ref{Lemma-way-of-construction}, and $p_1=2$. Hence by Lemma~\ref{Lemma-way-of-construction}, $A=\lim_{k\rightarrow \infty} p_k^{1/C_k}$ exists and $A=\min \mathcal{W}(c_k)$. Here we see that $p_3^{1/C_3}\leq A < (p_3+1)^{1/C_3}$.  By comparing the decimal expansions of $p_3^{1/C_3}$ and $(p_3+1)^{1/C_3}$, we obtain one of $A$ as in Proposition~\ref{Proposition-3k!}.    
\end{proof}

\begin{proof}[Proof of Proposition~\ref{Proposition-k!}] Let $c_k=k$ for every $k\in \mathbb{N}$. Assume the Riemann hypothesis. Similarly to the proof of Proposition~\ref{Proposition-k!}, by Theorem~\ref{Theorem-CMS} and using \texttt{ProvablePrimeQ[n]} in \textit{Mathematica}, we construct a sequence of prime numbers $(p_k)_{k\in \mathbb{N}}$ satisfying   \eqref{Condition-minimum} with $k_0=1$ in Lemma~\ref{Lemma-way-of-construction}, and 
\begin{gather*}
p_1=2, \quad p_2=5,\quad p_3=127,\quad p_4=(127)^4+22, \quad p_5=p_4^5+104,\quad p_6=p_5^6+700. 
\end{gather*}
By comparing the decimal expansions of $p_6^{1/6!}$ and $(p_6+1)^{1/6!}$, we can compute digits of $A$ to 254 decimal places.
\end{proof}

By Lemma~\ref{Lemma-way-of-construction}, we can also compute many examples of PRCs in $\subpartial \mathcal{W}(c_k)$. Indeed, by taking $p_1=2,3,5,7,11,\cdots$ and $k_0=1$, the set $\subpartial_L( \{A>1\colon \lfloor A^{3^{k!}}\rfloor \text{ is a PRF} \} )$ has elements of which decimal expansions are 
\[
    1.30529\cdots,\quad 1.45374\cdots,\quad 1.71299\cdots,\quad 1.91539\cdots,\quad 2.22949 \cdots.
\]
Note that these elements are algebraically independent by Corollary~\ref{Corollary-W(3k!)}.

\section*{Acknowledgement}
The first author was supported by JSPS KAKENHI Grant Number JP19J20878.

\bibliographystyle{amsalpha}
\bibliography{references_PR}
\end{document}